\documentclass[11pt]{article}
\usepackage[left=3.5cm,right=3.5cm,top=3cm,bottom=3cm]{geometry}
\usepackage{amsmath}
\usepackage{amssymb}
\usepackage{amsthm}
\usepackage{bm}
\usepackage{enumerate}
\usepackage{mathrsfs}
\usepackage{stmaryrd}

\usepackage{color}

\theoremstyle{definition}
\newtheorem{definition}{Definition}[section]
\newtheorem{notation}[definition]{Notation}
\newtheorem{example}[definition]{Example}
\newtheorem{assumption}[definition]{Assumption}

\theoremstyle{plain}
\newtheorem{proposition}[definition]{Proposition}
\newtheorem{theorem}[definition]{Theorem}
\newtheorem{lemma}[definition]{Lemma}
\newtheorem{corollary}[definition]{Corollary}

\theoremstyle{remark}
\newtheorem{remark}[definition]{Remark}

\numberwithin{equation}{section}

\newcommand{\E}{\mathbb{E}}
\newcommand{\N}{\mathbb{N}}
\renewcommand{\P}{\mathbb{P}}
\newcommand{\Q}{\mathbb{Q}}
\newcommand{\R}{\mathbb{R}}

\newcommand{\cB}{\mathcal{B}}
\newcommand{\cC}{\mathcal{C}}

\newcommand{\cE}{\mathcal{E}}
\newcommand{\cF}{\mathcal{F}}

\newcommand{\cH}{\mathcal{H}}

\newcommand{\cL}{\mathcal{L}}
\newcommand{\cM}{\mathcal{M}}
\newcommand{\cP}{\mathcal{P}}
\newcommand{\cQ}{\mathcal{Q}}
\newcommand{\cS}{\mathcal{S}}
\newcommand{\cU}{\mathcal{U}}

\newcommand{\cY}{\mathcal{Y}}

\newcommand{\rd}{\mathrm{d}}
\newcommand{\vp}{\varphi}
\newcommand{\pa}{\partial}
\newcommand{\bY}{\mathbf{Y}}
\newcommand{\sC}{\mathscr{C}}
\newcommand{\sD}{\mathscr{D}}
\newcommand{\bmz}{\bm{\zeta}}
\newcommand{\bme}{\bm{\eta}}
\newcommand{\frf}{\mathfrak{f}}
\newcommand{\tr}{\textrm{\normalfont{tr}}}
\newcommand{\lmin}{\lambda_{\mathrm{min}}}
\newcommand{\lmax}{\lambda_{\mathrm{max}}}
\newcommand{\Sm}{\cS_+^m}
\newcommand{\deta}{\dot{\eta}}
\newcommand{\Lipb}{\textrm{\textnormal{Lip}}_b}
\newcommand{\ver}[1]{{\left\vert\kern-0.25ex\left\vert\kern-0.25ex\left\vert #1 \right\vert\kern-0.25ex\right\vert\kern-0.25ex\right\vert}}
\newcommand{\simplex}{\Delta_{[0,T]}}

\newcommand{\opHol}{\frac{1}{p}\textrm{\textnormal{-H\"ol}}}
\newcommand{\tpHol}{\frac{2}{p}\textrm{\textnormal{-H\"ol}}}
\newcommand{\pvar}{p}
\newcommand{\ptvar}{\frac{p}{2}}
\newcommand{\pvarst}{p;[s,t]}
\newcommand{\pvartT}{p;[t,T]}
\newcommand{\pvarzT}{p;[0,T]}
\newcommand{\pvarzt}{p;[0,t]}
\newcommand{\pvarhr}{p;[h,r]}
\newcommand{\ptvarst}{\frac{p}{2};[s,t]}
\newcommand{\ptvartT}{\frac{p}{2};[t,T]}
\newcommand{\ptvarzT}{\frac{p}{2};[0,T]}
\newcommand{\ptvarzt}{\frac{p}{2};[0,t]}
\newcommand{\ptvarrt}{\frac{p}{2};[r,t]}
\newcommand{\onevartT}{1;[t,T]}
\newcommand{\onevarzt}{1;[0,t]}
\newcommand{\onevarhr}{1;[h,r]}
\newcommand{\onevarst}{1;[s,t]}
\newcommand{\onevarrt}{1;[r,t]}
\newcommand{\Cpvar}{\cC^{p\textrm{\textnormal{-var}}}}
\newcommand{\Cptvar}{\cC^{\frac{p}{2}\textrm{\textnormal{-var}}}}
\newcommand{\Cqtvar}{\cC^{\frac{q}{2}\textrm{\textnormal{-var}}}}
\newcommand{\gCpvar}{\cC^{0\textrm{\textnormal{,}}p\textrm{\textnormal{-var}}}}
\newcommand{\gCptvar}{\cC^{0\textrm{\textnormal{,}}\frac{p}{2}\textrm{\textnormal{-var}}}}

\DeclareMathOperator*{\esssup}{ess\,sup}

\DeclareMathOperator*{\argmin}{arg\,min}
\DeclareMathOperator*{\sgn}{sgn}

\begin{document}

\title{Pathwise Stochastic Control with Applications to Robust Filtering}
\author{Andrew L. Allan and Samuel N. Cohen\\
Mathematical Institute, University of Oxford\\
andrew.allan\hspace{-0.6pt}{\fontfamily{ptm}\selectfont @}\hspace{-0.4pt}maths.ox.ac.uk\\
samuel.cohen\hspace{-0.6pt}{\fontfamily{ptm}\selectfont @}\hspace{-0.4pt}maths.ox.ac.uk}
\date{\today}
\maketitle

\vspace{-12pt}

\begin{abstract}
We study the problem of pathwise stochastic optimal control, where the optimization is performed for each fixed realisation of the driving noise, by phrasing the problem in terms of the optimal control of rough differential equations. We investigate the degeneracy phenomenon induced by directly controlling the coefficient of the noise term, and propose a simple procedure to resolve this degeneracy whilst retaining dynamic programming. As an application, we use pathwise stochastic control in the context of stochastic filtering to construct filters which are robust to parameter uncertainty, demonstrating an original application of rough path theory to statistics.

\vspace{5pt}

\end{abstract}

\section{Introduction}

Stochastic optimal control is a classical optimization problem with numerous applications, from optimal liquidation and portfolio selection in mathematical finance to various problems in production planning, engineering and biology. Here one typically has a stochastic differential equation (SDE) of the form
\begin{equation}\label{eq:introcontSDE}
\rd X_s = b(X_s,\gamma_s)\,\rd s + \sigma(X_s,\gamma_s)\,\rd W_s, \qquad s \in [t,T],
\end{equation}
with an initial condition $X_t = x$, where $W$ is a Brownian motion and $\gamma$ is an adapted control process, and the goal is to minimize (resp.~maximize) a cost (resp.~reward) functional of the form
$$J(t,x;\gamma) = \E\bigg[\int_t^Tf(X_s,\gamma_s)\,\rd s + g(X_T)\bigg]$$
over all possible choices of the control $\gamma$. The resolution of this problem is by now well understood---two primary approaches being that of the Pontryagin stochastic maximum principle and Bellman's principle of optimality (or dynamic programming principle), which allows one to characterise the value function of the control problem, defined by $v(t,x) = \inf_\gamma J(t,x;\gamma)$, as the unique solution of a Hamilton--Jacobi--Bellman (HJB) partial differential equation (PDE).

In their 1998 paper \cite{LionsSouganidis1998}, Lions and Souganidis considered a variant of this problem, known as `pathwise stochastic control', where the optimization is performed pathwise. In other words, one considers controlling the solution of an equation of the form \eqref{eq:introcontSDE} for each individual realisation of the Brownian motion $W$. Moreover, they suggest that in this case the value function should satisty a `stochastic HJB equation'. Indeed, at least in the case when $\sigma$ does not depend on the control $\gamma$, if we pretend for the moment that the paths of $W$ were smooth, then, at least formally, the classical theory leads one to derive a stochastic PDE of the form
\begin{equation}\label{eq:stochHJBintro}
-\rd v - \inf_\gamma\big\{b\cdot\nabla v + f\big\}\,\rd t - \sigma\cdot\nabla v\,\rd W = 0,
\end{equation}
with $v(T,\cdot) = g$.

The notion of pathwise stochastic control actually goes back at least as far as the work of Davis and Burstein \cite{DavisBurstein1991,DavisBurstein1992}, who note that pathwise control is actually equivalent to the classical stochastic control setting if one allows for anticipative controls, leading to the conclusion that the difference between classical and pathwise control boils down to nonanticipativity of the controls. In this view, pathwise control can be thought of as performing optimal control with the benefit of complete knowledge of both the past and future realisations of the stochastic noise. On the other hand, as shown for instance by Rogers \cite{Rogers2007}, pathwise control can also be used to obtain duality results for classical (nonanticipative) stochastic control, thus providing an alternative approach for numerical computations.

Since pathwise control entails the optimization of a stochastic system path by path, it is natural to fix such an (arbitrary) path and proceed to analyse the resulting deterministic problem. This invites a pathwise interpretation of the stochastic integral appearing in the controlled dynamics. The strategy followed by Buckdahn and Ma \cite{BuckdahnMa2007} circumnavigates such a technical requirement, by instead employing a Doss--Sussmann-style transformation to convert the problem into a more standard setting of `wider-sense control problem', allowing them to establish their value function as the unique stochastic viscosity solution of the associated HJB equation. A more direct approach, avoiding such an ad hoc change of variables, requires one to utilize a pathwise approach to stochastic integration.

One such deterministic approach to integration against paths of low regularity is provided by rough path theory, introduced by Lyons \cite{Lyons1998}. The basic idea here is that the notion of integration can be extended in a consistent way to paths of lower regularity such that strong stability results concerning continuity of the integration map with respect to the driving `rough path' hold, but one requires extra information about the driving signal than is expressed in the path alone. Such paths, $\zeta$ say, must therefore be `enhanced' by a suitable `second order' process $\zeta^{(2)}$ which captures this missing information. The addition of the process $\zeta^{(2)}$ is equivalent to considering the L\'evy area of the path $\zeta$, and corresponds to the addition of the iterated integral $\int_0^\cdot\int_0^r\rd\zeta_s\otimes\rd\zeta_r$, but since this integral does not exist in the classical sense, its value must be postulated, rather than being uniquely determined by the original path $\zeta$.

We note in particular the more recent work of Diehl, Friz and Gassiat \cite{DiehlFrizGassiat2017}, which appears to be the first attempt to apply rough path theory to optimal control, in which the authors consider controlled dynamics of the form
\begin{equation}\label{eq:contdynzetaintro}
\rd X_s = b(X_s,\gamma_s)\,\rd s + \lambda(X_s)\,\rd\bmz_s,
\end{equation}
driven by a geometric rough path $\bmz$. They proceed to both obtain a version of Pontryagin's maximum principle and establish their value function as the unique solution of a `rough HJB equation', and moreover obtain a duality result for the corresponding nonanticipative stochastic control problem. Their results suggest that rough path theory is an ideal tool for the study of pathwise stochastic control---a notion that we will echo in the present work.

Notably however, the existing literature on pathwise control invariably focuses only on the case where the control process appears in the drift term, but does not appear in the coefficient of the noise term (or rough path). That is, the controlled dynamics considered are typically of the general form in \eqref{eq:contdynzetaintro}, where $\lambda$ is not allowed to depend on the control $\gamma$. As observed in Diehl et al.~\cite{DiehlFrizGassiat2017}, the pathwise control problem with full dynamic control (particularly control in the coefficient of the noise term), when stated in the obvious way, turns out to be degenerate, which explains the lack of results in this direction. An indication of this arises when one reruns the formal derivation that led to \eqref{eq:stochHJBintro} in the case where $\lambda$ depends on $\gamma$. In this case the resulting equation exhibits the Brownian motion $W$ \emph{inside} the infimum. At least heuristically, this corresponds to the ability to perfectly optimize over the path of the noise term, which, as we will see, is the source of the degeneracy.

\begin{example}[Insider trading]\label{naiveexample}
Let us suppose that an agent is trading a stock with price composed of a diffusive term with volatility $\sigma$, representing usual market uncertainty, and a deterministic path $\zeta$ which represents some additional information known only to the agent. We suppose that $\zeta \colon [0,T] \to \R$ is a continuous path which has infinite variation on any interval. Denoting the size of the agent's investment by $\gamma$ and their wealth process by $X$, we have the controlled dynamics\footnote{Let us suppose for the moment that we have employed a suitable notion of integration such that the integral against $\zeta$ is well-posed.}
\begin{equation}\label{eq:exampledynamics}
\rd X^{t,x,\gamma}_s = \gamma_s(\sigma\,\rd W_s + \rd\zeta_s), \qquad s \in [t,T],
\end{equation}
with $X^{t,x,\gamma}_t = x$, where $W$ is a standard Brownian motion and $x$ is the agent's initial wealth. We assume that at each time the agent can only hold a finite amount of stock. More precisely, we impose that $\gamma$ takes values in the finite interval $[-\varepsilon,\varepsilon]$, for some $\varepsilon > 0$. The agent's expected terminal wealth is given by the value function
\begin{equation}\label{eq:examplevaluefunc}
v(t,x) = \sup_{\gamma}\E\big[X^{t,x,\gamma}_T\big],
\end{equation}
where the supremum is taken over the collection of progressively measurable $[-\varepsilon,\varepsilon]$-valued processes. In order to study the nature of this problem, let us approximate $\zeta$ by a smooth function $\eta$. The dynamics \eqref{eq:exampledynamics} are then approximated by
$$\rd X^{t,x,\gamma,\eta}_s = \gamma_s(\sigma\,\rd W_s + \deta_s\,\rd s),$$
where $\deta$ denotes the derivative of $\eta$. The resulting control problem is classical. The associated HJB equation is given by
$$\frac{\pa v^\eta}{\pa t} + \sup_{\gamma \in [-\varepsilon,\varepsilon]}\bigg\{\frac{1}{2}\gamma^2\sigma^2\frac{\pa^2v^\eta}{\pa x^2} + \gamma\deta_t\frac{\pa v^\eta}{\pa x}\bigg\} = 0,$$
with terminal condition $v^\eta(T,x) = x$. The solution $v^\eta$ is seen to be
\begin{equation}\label{eq:exHJBsoln}
v^\eta(t,x) = x + \varepsilon\int_t^T|\deta_s|\,\rd s,
\end{equation}
and we infer that the optimal control $\gamma^\ast$ is given by
\begin{equation}\label{eq:exoptcontrol}
\gamma^\ast_t = \varepsilon\sgn(\deta_t).
\end{equation}
Suppose now that we were to repeatedly refine the approximation $\eta$ so that it better captures the fast fluctuations of $\zeta$. Even without giving a precise definition of what we might mean by the limit as $\eta \to \zeta$, it is clear that the solution \eqref{eq:exHJBsoln} should diverge to infinity in this limit. In other words, the original value function, as defined in \eqref{eq:examplevaluefunc}, should be simply given by $v(t,x) = \infty$ whenever $t < T$. Thus, we infer that our original control problem, with the infinite variation signal $\zeta$, is degenerate.
\end{example}

The phenomenon exhibited here is typical for such control problems, where we attempt to control the coefficient of the infinite variation term in the controlled dynamics. The problem in the previous example is that, in contrast to a classical stochastic setting, since the controller can `see' the path $\zeta$ in advance, they can choose controls $\gamma$ with very small, but extremely quick, fluctuations, which allow the solution $X$ to take full advantage of the infinite variation of $\zeta$. Indeed, notice that the sign of the optimal control in \eqref{eq:exoptcontrol} changes at the same rate as that of $\deta$, which varies `infinitely quickly' in the limit as $\eta \to \zeta$.

To paraphrase Diehl et al.~\cite{DiehlFrizGassiat2017}, if the coefficient of the driving signal has enough dependence on the control, and this signal has unbounded variation on any interval, then the controller can drive the solution to reach any point instantly whilst incurring an arbitrarily low cost.

In the current work we investigate this degeneracy phenomenon in more detail, and see how it may be resolved by introducing an artificial cost to penalise the variation of the controls. We will see how this cost may be chosen to ensure that a dynamic programming principle is retained, thus allowing one to recover a setting comparable to that of \cite{DiehlFrizGassiat2017}. As an extension of that paper, we proceed to consider pathwise control with unbounded cost functions and, by obtaining locally uniform bounds on the controls, establish the value function as the unique solution of a rough HJB equation.

Stochastic filtering concerns the problem of estimating the current state of a hidden process from noisy observations, and itself has widespread and important applications, from finance and biology to engineering, defence and aerospace. In their paper \cite{CrisanDiehlFrizOberhauser2013}, Crisan, Diehl, Friz and Oberhauser used rough path theory to resolve an existing open problem in the theory of `robust' stochastic filtering, by establishing continuity of a large class of stochastic filters with respect to the observation path, by first enhancing it by its L\'evy area. The second contribution of the current work is to consider an application of pathwise control to another kind of `robust' filtering, namely robustness of the filter with respect to model uncertainty.

Although classical filters are generally known to perform well under perfect knowledge of the system dynamics, they are typically very sensitive to modelling errors. Thus, the problem of robust filtering, in this sense, has attracted a great deal of interest; see the discussion at the beginning of Section~\ref{SecRobustFiltViaNonlinExp}. In \cite{AllanCohen2019} the authors constructed such a robust filter, the calculation of which involves the derivation of a pathwise stochastic control problem. In that setting the control terms did not appear (in any crucial way) in the coefficient of the driving noise in the controlled dynamics. Similarly to the approach in Buckdahn and Ma \cite{BuckdahnMa2007}, a change of variables could therefore be used to `hide' the rough noise term in the drift coefficient, thus recovering a more classical optimal control setting. In the current work we aim to significantly extend the theoretical results of \cite{AllanCohen2019}, which will require us to be able to handle full control of the dynamics. As in Crisan et al.~\cite{CrisanDiehlFrizOberhauser2013}, it will be useful to consider the observation process as a rough path, by first enhancing it by its L\'evy area.

It is our hope that the following exposition will be of interest to readers familiar with rough path analysis, but also accessible to those without a working knowledge of the subject. Accordingly, we begin Section~\ref{SecRoughPathPrelim} with a brief recall of the necessary technical preliminaries, and then present some new results for rough differential equations in the setting of optimal control. In Section~\ref{SecPathwiseControl} we discuss some alternative reformulations of the pathwise control problem with the aim to resolve the degeneracy issue. We provide a rigorous treatment of the resulting unbounded control problem, and illustrate the ideas with some simple examples. In Section~\ref{SecApplicFilter} we turn our attention to robust stochastic filtering. Our approach leads naturally to a pathwise optimal control problem and, despite the nonlinearities inherited from the classical filtering equations, we will proceed to characterise the associated value function as the solution of a rough HJB equation.

\section{Rough path preliminaries}\label{SecRoughPathPrelim}

We would like to consider an $\R^m$-valued process $X$ which, for each choice of control $\gamma \colon [0,T] \to \R^k$, satisfies an equation of the form
\begin{equation}\label{eq:roughdynamicsintro}
\rd X_s = b(X_s,\gamma_s)\,\rd s + \lambda(X_s,\gamma_s)\,\rd \zeta_s, \qquad s \in [0,T],
\end{equation}
where $\zeta$ is a continuous (deterministic) $\R^d$-valued path of infinite variation.

Suppose that $b$ and $\lambda$ are Lipschitz continuous. In the case when $\gamma$ is of finite variation and $\lambda$ does not depend on the solution $X$, the integral against $\zeta$ then exists in the classical Riemann--Stieltjes sense by integration by parts (see e.g.~Theorem~1.2.3 in Stroock \cite{Stroock2011}). The equation \eqref{eq:roughdynamicsintro} then has a unique solution, and moreover the solution map from the driver $\zeta$ to the corresponding solution $X^\zeta$ is continuous with respect to the supremum norm. In fact, in this case all the results of the next section can be reproduced without any reference to rough path theory, or any other such sophisticated machinery.

On the other hand, in the general case when $\lambda$ depends on $X$, the integral against $\zeta$ in \eqref{eq:roughdynamicsintro} does not even exist in the Riemann--Stieltjes sense. Moreover, even if $\zeta$ were smooth, the solution map $\zeta \mapsto X^\zeta$ is known to lack continuity, which later would be fatal to the derivation of our HJB equation.

A deterministic approach to integration against very general classes of signals is provided by rough path theory, which moreover allows the continuity property mentioned above to be recovered. As mentioned in the introduction, the key here is, rather than to simply integrate against the path $\zeta$, to first enhance $\zeta$ by a suitable `second order' process $\zeta^{(2)}$, which contains the missing information required to construct the so-called `rough integral' against the enhanced path $\bmz := (\zeta,\zeta^{(2)})$. There are by now a number of monographs on this subject, such as Friz and Hairer \cite{FrizHairer2014} and Friz and Victoir \cite{FrizVictoir2010}.

The language of rough path theory is typically written either in terms of the $\frac{1}{p}$-H\"older regularity of paths, or in terms of their $p$-variation. When working only with continuous paths (as we shall), these two notions of regularity are more or less equivalent (see Chapter~5 in \cite{FrizVictoir2010} for precise details), and the theory may be built up in an almost identical fashion using either notion. In the current work it will turn out to be necessary to work primarily with $p$-variation norms. On the other hand, in the proof of Proposition~\ref{proproughestimates} below we will make use of the marginally better control on the regularity of paths over small time intervals provided by restricting to $\frac{1}{p}$-H\"older rough paths. We shall therefore make use of both these notions of regularity.

\subsection{Notation}

Throughout, we will consider a finite time interval $[0,T]$, and write $\simplex := \{(s,t) : 0 \leq s \leq t \leq T\}$ for the standard 2-simplex. For any path $\zeta$ on $[0,T]$ we define the path increment $\zeta_{s,t} := \zeta_t - \zeta_s$, and write $\|\zeta\|_\infty := \sup_{s \in [0,T]}|\zeta_s|$ for the supremum norm. We will also make use of the following function spaces. We write
\begin{itemize}
\item $\cL(\R^d;\R^m)$ for the space of linear maps from $\R^d$ to $\R^m$,
\item $\Lipb$ for the space of bounded Lipschitz functions $b \colon \R^m \times \R^k \to \R^m$,
\item $C^n_b$ ($n \in \N$) for the space of $n$ times continuously differentiable (in the Fr\'echet sense) functions $\lambda \colon \R^m \times \R^k \to \cL(\R^d;\R^m)$ such that $\lambda$ and all its derivatives up to order $n$ are uniformly bounded,
\item $\Cpvar = \Cpvar([0,T];\R^k)$ for the space of $\R^k$-valued continuous paths of finite $p$-variation, that is, continuous paths $\gamma$ such that the seminorm
$$\|\gamma\|_{\pvar} := \bigg(\sup_{\cP}\sum_{[s,t] \in \cP}|\gamma_{s,t}|^p\bigg)^{\hspace{-2pt}\frac{1}{p}} < \infty,$$
where the supremum is taken over all partitions $\cP$ of the interval $[0,T]$,
\item $\gCpvar = \gCpvar([0,T];\R^k)$ for the closure of smooth paths from $[0,T] \to \R^k$ with respect to the $p$-variation seminorm.
\end{itemize}

For $p \in [2,3)$ we write $\sC^p = \sC^p([0,T];\R^d)$ for the space of $\R^d$-valued $\frac{1}{p}$-H\"older rough paths, that is, pairs $\bmz = (\zeta,\zeta^{(2)})$, where the path $\zeta \colon [0,T] \to \R^d$ and its `enhancement' $\zeta^{(2)} \colon \simplex \to \R^d \otimes \R^d$ satisfy certain algebraic and analytical constraints, namely Chen's relation\footnote{Here $\otimes$ is just the standard tensor product from $\R^d \times \R^d$ to $\R^d \otimes \R^d \simeq \R^{d \times d}$.},
\begin{equation*}
\zeta^{(2)}_{s,t} = \zeta^{(2)}_{s,r} + \zeta^{(2)}_{r,t} + \zeta_{s,r} \otimes \zeta_{r,t},
\end{equation*}
which is assumed to hold for all times $s \leq r \leq t$, as well as the condition that
\begin{gather*}
\ver{\bmz}_{\opHol} := \|\zeta\|_{\opHol} + \big\|\zeta^{(2)}\big\|_{\tpHol} < \infty,\\
\text{where} \qquad \|\zeta\|_{\opHol} := \sup_{s \neq t \in [0,T]}\frac{|\zeta_{s,t}|}{|t - s|^{\frac{1}{p}}} \qquad \text{and} \qquad \big\|\zeta^{(2)}\big\|_{\tpHol} := \sup_{s \neq t \in [0,T]}\frac{\big|\zeta^{(2)}_{s,t}\big|}{|t - s|^{\frac{2}{p}}}.
\end{gather*}
The enhanced path $\bmz$ is sometimes referred to as the `lift' of $\zeta$. We also define
\begin{align*}
\big\|\zeta^{(2)}\big\|_{\ptvar} &:= \bigg(\sup_{\cP}\sum_{[s,t] \in \cP}\big|\zeta^{(2)}_{s,t}\big|^{\frac{p}{2}}\bigg)^{\hspace{-2pt}\frac{2}{p}},\\
\ver{\bmz}_p &:= \|\zeta\|_{p} + \big\|\zeta^{(2)}\big\|_{\frac{p}{2}}.
\end{align*}
We will sometimes write e.g.~$\|\zeta\|_{p;[s,t]}$ for the $p$-variation of $\zeta$ over the subinterval $[s,t]$.

As we are working on the time interval $[0,T]$, it is straightforward to see that any rough path $\bmz = (\zeta,\zeta^{(2)}) \in \sC^p$ satisfies $\|\zeta\|_{\pvar} \leq \|\zeta\|_{\opHol}T^{\frac{1}{p}}$ and $\|\zeta^{(2)}\|_{\ptvar} \leq \|\zeta^{(2)}\|_{\tpHol}T^{\frac{2}{p}}$, which in particular implies that $\ver{\bmz}_p < \infty$ for any $\bmz \in \sC^p$.

We introduce the induced rough path metrics\footnote{The `metrics' $\varrho_{\opHol}$, $\varrho_p$ do not distinguish between constants, but $\sC^p$ does become a complete metric space when endowed with the metric $(\bme,\bmz) \mapsto |\eta_0 - \zeta_0| + \varrho_{\opHol}(\bme,\bmz)$.} given, for rough paths $\bme = (\eta,\eta^{(2)})$ and $\bmz = (\zeta,\zeta^{(2)})$, by
\begin{align*}
\varrho_{\opHol}(\bme,\bmz) &:= \|\eta - \zeta\|_{\opHol} + \big\|\eta^{(2)} - \zeta^{(2)}\big\|_{\tpHol},\\
\varrho_p(\bme,\bmz) &:= \|\eta - \zeta\|_{\pvar} + \big\|\eta^{(2)} - \zeta^{(2)}\big\|_{\ptvar}.
\end{align*}

As can be readily checked, any \emph{smooth} path $\zeta \colon [0,T] \to \R^d$ can be `lifted' in a canonical way to a rough path $\bmz = (\zeta,\zeta^{(2)})$ by enhancing it with the integral
\begin{equation}\label{eq:iterintegrals}
\zeta^{(2)}_{s,t} = \int_s^t\zeta_{s,r}\otimes\rd\zeta_r.
\end{equation}
On the other hand, for a general $\frac{1}{p}$-H\"older continuous path $\zeta$, the integral in \eqref{eq:iterintegrals} does not exist in the classical sense. In this case the value of this integral is postulated by the enhancement $\zeta^{(2)}$, which in practice is often constructed using stochastic integration.

Later we will also consider the space of \emph{geometric} rough paths $\sC_g^{0,p} \subset \sC^p$, defined as the closure of canonical lifts of smooth paths with respect to $\varrho_{\opHol}$. For example, when $\zeta$ is a semimartingale and the integral in \eqref{eq:iterintegrals} is defined using Stratonovich integration, the resulting lift turns out to be a (random) geometric rough path. This property of being well approximated by smooth paths allows one to make sense of solutions to a wide class of rough ODEs and PDEs---we will see an example of this in Definition~\ref{defnsolnroughHJB} below.

\subsection{Rough integration}

We now define a suitable class of integrands for rough integration. Given a rough path $\bmz \in \sC^p$, we define the space of controlled rough paths (in the sense of Gubinelli \cite{Gubinelli2004}), which we denote by $\sD_\zeta^{p} = \sD_\zeta^{p}([0,T];\R^m)$, consisting of pairs of paths
$$(X,X') \in \Cpvar([0,T];\R^m) \times \Cpvar([0,T];\cL(\R^d;\R^m))$$
such that the remainder term $R^X$, given by
$$R^X_{s,t} := X_{s,t} - X'_s\zeta_{s,t},$$
satisfies $\|R^X\|_{\frac{p}{2}} < \infty$. Here $X'$ is called the Gubinelli derivative of $X$ (with respect to $\zeta$). Equipped with the norm $(X,X') \mapsto |X_0| + |X'_0| + \|X'\|_p + \|R^X\|_{\frac{p}{2}}$, the space $\sD_\zeta^{p}$ is a Banach space.

\begin{remark}
As our main interest is in the optimal control of the solution $X$ to \eqref{eq:roughdynamicsintro}, the notion that $X$ is `controlled' by $\zeta$ introduces a possible source of confusion, but our use of the term should always be clear from the context.
\end{remark}

\begin{proposition}[Proposition~2.6 in \cite{FrizZhang2018}]\label{proproughintegral}
Let $\bmz = (\zeta,\zeta^{(2)}) \in \sC^p([0,T];\R^d)$, and let $(X,X') \in \sD^p_\zeta([0,T];\cL(\R^d;\R^m))$ be a controlled rough path. Then the limit
\begin{equation*}
\int_0^TX_r\,\rd\bmz_r := \lim_{|\cP| \to 0}\sum_{[s,t] \in \cP}X_s\zeta_{s,t} + X'_s\zeta^{(2)}_{s,t}
\end{equation*}
exists\footnote{Strictly speaking, in making precise sense of the product $X'_s\zeta^{(2)}_{s,t}$, we use the natural identification of $\cL(\R^d;\cL(\R^d;\R^m))$ with $\cL(\R^d \otimes \R^d;\R^m)$.}, where the limit is taken over any sequence of partitions $\cP$ of the interval $[0,T]$ such that the mesh size $|\cP| \to 0$. This limit (which does not depend on the choice of sequence of partitions) is called the rough integral of $X$ against $\bmz$.

Moreover, for any $0 \leq s < t \leq T$, we have the estimate
\begin{align}
\bigg|\int_s^tX_r\,\rd\bmz_r& - X_s\zeta_{s,t} - X'_s\zeta^{(2)}_{s,t}\bigg|\nonumber\\
&\leq C_p\Big(\big\|R^X\big\|_{\ptvarst}\|\zeta\|_{\pvarst} + \|X'\|_{\pvarst}\big\|\zeta^{(2)}\big\|_{\ptvarst}\Big),\label{eq:roughintbound}
\end{align}
where the constant $C_p$ depends only on $p$.
\end{proposition}

\subsection{Rough differential equations with controls}

For a given $p \in [2,3)$, rough path $\bmz = (\zeta,\zeta^{(2)}) \in \sC^p([0,T];\R^d)$ and control function $\gamma \in \Cptvar([0,T];\R^k)$, we consider the rough differential equation (RDE)
\begin{equation}\label{eq:RDEforX}
\rd X_s = b(X_s,\gamma_s)\,\rd s + \lambda(X_s,\gamma_s)\,\rd \bmz_s, \qquad s \in [0,T],
\end{equation}
controlled (in the sense of optimal control) by $\gamma$, with $X_0 = x \in \R^m$, where the second term on the right-hand side is interpreted as a rough integral against $\bmz$.

The main element that takes us outside the standard RDE setting is the appearance of the control $\gamma$ in the coefficients. Note however that, since $\gamma \in \Cptvar$, it is immediately controlled by $\zeta$ with $\gamma' = 0$, so that $(\gamma,0) \in \sD_\zeta^{p}([0,T];\R^k)$. Then, provided that $\lambda \in C^2_b$, for any $(X,X') \in \sD_\zeta^{p}([0,T];\R^m)$, the composition $\lambda(X,\gamma)$ can also be interpreted as being controlled by $\zeta$, with Gubinelli derivative given by
\begin{equation}\label{eq:GubderivlambdaXgamma}
\lambda(X,\gamma)' = D_x\lambda(X,\gamma)X',
\end{equation}
where $D_x\lambda$ is the Fr\'echet derivative of $\lambda$ in its first argument.

\begin{lemma}\label{lemmapvarpartitionbound}
For some $n \geq 1$, let $0 = t_0 < t_1 < \ldots < t_{n-1} < t_n = T$, be a partition of the interval $[0,T]$. Then, for any path $X$, one has that
$$\|X\|_{\pvarzT} \leq n\bigg(\sum_{i=1}^n\|X\|_{p;[t_{i-1},t_i]}^p\bigg)^{\hspace{-2pt}\frac{1}{p}}.$$
\end{lemma}

\begin{proof}
Let $0 = s_0 < s_1 < \ldots < s_{N-1} < s_N = T$ be another partition of the interval $[0,T]$. We can label the union of these two partitions in two different ways as follows. We can either write
\begin{align*}
s_{j-1} = t^j_0 < t^j_1 < \ldots < t^j_{n_j} = s_j \quad &\text{for each} \quad j = 1, \ldots, N,\\
\text{or} \quad t_{i-1} = s^i_0 < s^i_1 < \ldots < s^i_{N_i} = t_i \quad &\text{for each} \quad i = 1, \ldots, n,
\end{align*}
where, crucially, $n_j \leq n$ for every $j$. We have
\begin{align*}
\sum_{j=1}^N\big|X_{s_j} - X_{s_{j-1}}\hspace{-1pt}\big|^p &\leq \sum_{j=1}^N\bigg(\sum_{i=1}^{n_j}\big|X_{t^j_i} - X_{t^j_{i-1}}\hspace{-1pt}\big|\bigg)^{\hspace{-2pt}p} \leq n^p\sum_{j=1}^N\sum_{i=1}^{n_j}\big|X_{t^j_i} - X_{t^j_{i-1}}\hspace{-1pt}\big|^p\\
&= n^p\sum_{i=1}^{n}\sum_{j=1}^{N_i}\big|X_{s^i_j} - X_{s^i_{j-1}}\hspace{-1pt}\big|^p \leq n^p\sum_{i=1}^{n}\|X\|_{p;[t_{i-1},t_i]}^p.
\end{align*}
The result then follows from taking the supremum over all possible partitions $s_0 < s_1 < \ldots < s_N$ of the interval $[0,T]$.
\end{proof}

\begin{proposition}\label{proproughestimates}
Let $b \in \text{Lip}_b$ and $\lambda,\psi \in C^2_b$. For some $p \in [2,3)$ and $L > 0$, let $\bmz \in \sC^p$ such that $\ver{\bmz}_{\opHol} \leq L$, and suppose that $X$ satisfies the RDE \eqref{eq:RDEforX} with $X' = \lambda(X,\gamma)$, for some $\gamma \in \Cptvar$. We have the following estimates:
\begin{enumerate}[(i)]
\item $\|\psi(X,\gamma)'\|_p \leq C_{\lambda,\psi,p}\big(\|X\|_p + \|\gamma\|_{\frac{p}{2}}\big)$,
\item $\big\|R^{\psi(X,\gamma)}\big\|_{\frac{p}{2}} \leq C_{\psi,p}\Big(\|X\|_p^2 + \big\|R^X\big\|_{\frac{p}{2}} + \|\gamma\|_{\frac{p}{2}}\Big)$,
\item $\|X\|_p \leq C_{b,\lambda,p,T,L}\big(1 + \|\gamma\|_{\frac{p}{2}}^{1 + p}\big)$,
\item $\big\|R^X\big\|_{\frac{p}{2}} \leq C_{b,\lambda,p,T,L}\big(1 + \|\gamma\|_{\frac{p}{2}}^{2 + p}\big)$,
\end{enumerate}
where in each case the constant $C$ depends only on the variables indicated.
\end{proposition}

\begin{proof}
The first two estimates follow from standard arguments, noting that the Gubinelli derivative of $\psi(X,\gamma)$ is given by $\psi(X,\gamma)' = D_x\psi(X,\gamma)\lambda(X,\gamma)$. Let us therefore turn our attention to the proof of (iii). In the following the symbol $\lesssim$ shall denote inequality up to a multiplicative constant depending only on $b,\lambda,p,T$ and $L$.

Let $[s,t] \subseteq [0,T]$. We then have
\begin{align*}
\big|R^X_{s,t}\big| &= \big|X_{s,t} - X'_s\zeta_{s,t}\big|\\
&\leq \bigg|\int_s^t\lambda(X_r,\gamma_r)\,\rd\bmz_r - \lambda(X_s,\gamma_s)\zeta_{s,t} - \lambda(X,\gamma)'_s\zeta^{(2)}_{s,t}\bigg|\\
&\qquad + \bigg|\int_s^tb(X_r,\gamma_r)\,\rd r\bigg| + \big|\lambda(X,\gamma)'_s\zeta^{(2)}_{s,t}\big|\\
&\lesssim \big\|R^{\lambda(X,\gamma)}\big\|_{\frac{p}{2};[s,t]}\|\zeta\|_{p;[s,t]} + \|\lambda(X,\gamma)'\|_{p;[s,t]}\big\|\zeta^{(2)}\big\|_{\frac{p}{2};[s,t]} + |t - s| + \big|\zeta^{(2)}_{s,t}\big|,
\end{align*}
where we applied \eqref{eq:roughintbound} to obtain the last line. It follows that for a given interval $I \subseteq [0,T]$ of length $|I|$,
$$\big\|R^X\big\|_{\frac{p}{2};I} \lesssim \big\|R^{\lambda(X,\gamma)}\big\|_{\frac{p}{2};I}\|\zeta\|_{p;I} + \|\lambda(X,\gamma)'\|_{p;I}\big\|\zeta^{(2)}\big\|_{\frac{p}{2};I} + |I| + \big\|\zeta^{(2)}\big\|_{\frac{p}{2};I}.$$
Applying the estimates in (i) and (ii) with $\psi = \lambda$, we obtain
\begin{align*}
\big\|R^X\big\|_{\frac{p}{2};I} \leq C_1\Big(&\big(\|X\|_{p;I}^2 + \big\|R^X\big\|_{\frac{p}{2};I} + \|\gamma\|_{\frac{p}{2};I}\big)\|\zeta\|_{p;I}\\
&\quad + \big(1 + \|X\|_{p;I}^2 + \|\gamma\|_{\frac{p}{2};I}\big)\big\|\zeta^{(2)}\big\|_{\frac{p}{2};I} + |I| + \big\|\zeta^{(2)}\big\|_{\frac{p}{2};I}\Big)
\end{align*}
for some constant $C_1$ (which only depends on $b,\lambda$ and $p$).

Since $\|\zeta\|_{p;I} \leq \|\zeta\|_{\opHol}|I|^{\frac{1}{p}} \leq L|I|^{\frac{1}{p}}$, there exists some $r > 0$ (depending only on $p,L$ and $C_1$) sufficiently small such that
\begin{equation}\label{eq:C1zetalesshalf}
C_1\|\zeta\|_{p;I} \leq \frac{1}{2}
\end{equation}
whenever $|I| \leq r$. It is enough to prove the result for $T \leq r$, since one can then extend the result to any larger $T$ using Lemma~\ref{lemmapvarpartitionbound}. We will therefore assume that $T \leq r$, so that \eqref{eq:C1zetalesshalf} holds for all intervals $I$ under consideration. We then deduce that
\begin{equation}\label{eq:RXboundX2gamma}
\big\|R^X\big\|_{\frac{p}{2};I} \lesssim \big(\|X\|_{p;I}^2 + \|\gamma\|_{\frac{p}{2};I}\big)\|\zeta\|_{p;I} + \big(1 + \|X\|_{p;I}^2 + \|\gamma\|_{\frac{p}{2};I}\big)\big\|\zeta^{(2)}\big\|_{\frac{p}{2};I} + |I|.
\end{equation}
From the basic estimate
\begin{equation}\label{eq:XboundzetaRX}
\|X\|_{p;I} \lesssim \|\zeta\|_{p,I} + \|R^X\|_{\frac{p}{2};I},
\end{equation}
we then have that
\begin{equation}\label{eq:XboundC2X2}
\|X\|_{p;I} \leq C_2\big(1 + \|\gamma\|_{\frac{p}{2};I}\big)\Big(\|\zeta\|_{p;I} + \big\|\zeta^{(2)}\big\|_{\frac{p}{2};I} + |I|\Big) + C_2\|X\|_{p;I}^2
\end{equation}
for some constant $C_2$ (depending on $b,\lambda,p$ and $L$). From here, we aim to infer an estimate which holds on small subintervals, and then use Lemma~\ref{lemmapvarpartitionbound} to paste such subintervals together to obtain an estimate which holds on the entire interval $[0,T]$.

It follows from above that, if $C_2\|X\|_{p;I} \leq \frac{1}{2}$, then
\begin{equation*}
\|X\|_{p;I} \leq 2C_2\big(1 + \|\gamma\|_{\frac{p}{2};I}\big)\Big(\|\zeta\|_{p;I} + \big\|\zeta^{(2)}\big\|_{\frac{p}{2};I} + |I|\Big).
\end{equation*}
Let $t^\ast = \sup\{t \in [0,T] : C_2\|X\|_{p;I} \leq \frac{1}{2} \text{ whenever } |I| \leq t\}$. If $t^\ast = T$ then we are done. Otherwise, let $I$ be an interval such that $|I| = t^\ast$ and $C_2\|X\|_{p;I} = \frac{1}{2}$. Then
\begin{align*}
\frac{1}{2C_2} = \|X\|_{p;I} &\leq 2C_2\big(1 + \|\gamma\|_{\frac{p}{2};I}\big)\Big(\|\zeta\|_{p;I} + \big\|\zeta^{(2)}\big\|_{\frac{p}{2};I} + |I|\Big)\\
&\leq 2C_2\big(1 + \|\gamma\|_{\frac{p}{2};[0,T]}\big)\Big(\|\zeta\|_{\opHol}(t^\ast)^{\frac{1}{p}} + \big\|\zeta^{(2)}\big\|_{\tpHol}(t^\ast)^{\frac{2}{p}} + t^\ast\Big),
\end{align*}
and we deduce that
$$\frac{1}{t^\ast} \lesssim 1 + \|\gamma\|_{\frac{p}{2};[0,T]}^p.$$
The interval $[0,T]$ can be partitioned into $n := \lceil T/t^\ast \rceil$ subintervals $I$ of length at most $t^\ast$, on each of which we have $\|X\|_{p;I} \leq \frac{1}{2C_2}$. From Lemma~\ref{lemmapvarpartitionbound}, we obtain the bound $\|X\|_{p;[0,T]} \lesssim n^{1 + \frac{1}{p}}$, where $n \leq 1 + T/t^\ast \lesssim 1 + \|\gamma\|_{\frac{p}{2};[0,T]}^p$, and the estimate in (iii) follows.

Substituting \eqref{eq:XboundzetaRX} into \eqref{eq:RXboundX2gamma}, we have
$$\big\|R^X\big\|_{\frac{p}{2};I} \leq C_3\big(1 + \|\gamma\|_{\frac{p}{2};I}\big)\Big(\|\zeta\|_{p;I} + \big\|\zeta^{(2)}\big\|_{\frac{p}{2};I} + |I|\Big) + C_3\big\|R^X\big\|_{\frac{p}{2};I}^2$$
for some new constant $C_3$. This equation is of the same form as \eqref{eq:XboundC2X2}. We can thus apply exactly the same argument as above to deduce the estimate in (iv).
\end{proof}

The results of Theorem~\ref{thmsolnofRDE} and Proposition~\ref{proproughstability} below are new in this setting due to the inclusion of the control function $\gamma$, particularly in the controlled path setting of Gubinelli with path regularity measured in $p$-variation, but they are based upon standard results, so we shall postpone their proofs to the Appendix.

\begin{theorem}\label{thmsolnofRDE}
Let $b \in \Lipb$, $\lambda \in C^3_b$ and $\bmz \in \sC^p$. For any $x \in \R^m$ and any $\gamma \in \Cptvar$, there exists a unique solution $(X,X') \in \sD_\zeta^{p}$ to the RDE
\begin{equation}\label{eq:RDEforXthm}
X_t = x + \int_0^tb(X_s,\gamma_s)\,\rd s + \int_0^t\lambda(X_s,\gamma_s)\,\rd \bmz_s, \qquad t \in [0,T],
\end{equation}
such that $X' = \lambda(X,\gamma)$, where $\lambda(X,\gamma)$ is interpreted as a controlled rough path with Gubinelli derivative given by \eqref{eq:GubderivlambdaXgamma}.
\end{theorem}

\begin{proposition}\label{proproughstability}
Let $b \in \Lipb$, $\lambda \in C^3_b$, $\gamma,\vartheta \in \Cptvar$ and $\bme,\bmz \in \sC^p$ with $\ver{\bme}_{\opHol} \leq L$, $\ver{\bmz}_{\opHol} \leq L$. Let $(X,X') = (X,\lambda(X,\gamma)) \in \sD^p_\eta$ (resp.~$(Y,Y') = (Y,\lambda(Y,\vartheta)) \in \sD^p_\zeta$) be the unique solution of the RDE \eqref{eq:RDEforXthm} controlled by $\gamma$ (resp.~$\vartheta$) and driven by $\bme$ (resp.~$\bmz$) with the initial condition $x$ (resp.~$y$). Suppose that $\|\gamma\|_{\frac{p}{2}},\|\vartheta\|_{\frac{p}{2}} \leq M$ for some $M > 0$. Then
\begin{equation}\label{eq:roughstability}
\|X' - Y'\|_p + \big\|R^X - R^Y\big\|_{\frac{p}{2}} \leq C\Big(|x - y| + \|\gamma - \vartheta\|_\infty + \|\gamma - \vartheta\|_{\frac{p}{2}} + \varrho_p(\bme,\bmz)\Big).
\end{equation}
Moreover, given $\psi \in C^3_b$, we have
\begin{equation}\label{eq:roughstabilityintegral}
\bigg\|\int_0^\cdot\psi(X_s,\gamma_s)\,\rd\bme_s - \int_0^\cdot\psi(Y_s,\vartheta_s)\,\rd\bmz_s\bigg\|_p \leq C'\Big(|x - y| + \|\gamma - \vartheta\|_\infty + \|\gamma - \vartheta\|_{\frac{p}{2}} + \varrho_p(\bme,\bmz)\Big).
\end{equation}
Here the constants $C,C'$ depend on $b,\lambda,p,T,L$ and $M$, and $C'$ also depends on $\psi$.
\end{proposition}

\section{Pathwise optimal control}\label{SecPathwiseControl}

\subsection{Avoiding degeneracy}

Our set-up is the following. We fix a geometric rough path $\bmz = (\zeta,\zeta^{(2)}) \in \sC_g^{0,p}([0,T];\R^d)$ such that $\ver{\bmz}_{\opHol} \leq L$ for some $p \in [2,3)$ and $L > 0$. We consider, for each $\gamma \in \Cptvar([0,T];\R^k)$, the controlled dynamics
\begin{equation}\label{eq:RDEcontroldynamics}
\rd X^{t,x,\gamma}_s = b(X^{t,x,\gamma}_s,\gamma_s)\,\rd s + \lambda(X^{t,x,\gamma}_s,\gamma_s)\,\rd\bmz_s, \qquad \quad X^{t,x,\gamma}_t = x,
\end{equation}
driven by $\bmz$. We then consider the control problem with value function given by
\begin{equation}\label{eq:valuefuncnaive}
v(t,x) := \inf_{\gamma \in \Cptvar}J(t,x;\gamma)
\end{equation}
for $(t,x) \in [0,T] \times \R^m$, where the cost functional $J$ is defined as
\begin{equation}\label{eq:costfuncJ}
J(t,x;\gamma) := \int_t^Tf(X^{t,x,\gamma}_s,\gamma_s)\,\rd s + \int_t^T\psi(X^{t,x,\gamma}_s,\gamma_s)\,\rd\bmz_s + g(X^{t,x,\gamma}_T).
\end{equation}
Here $f \colon \R^m \times \R^k \to \R$, $\psi \colon \R^m \times \R^k \to \cL(\R^d;\R)$ and $g \colon \R^m \to \R$.

\begin{lemma}\label{lemmapsiintbound}
Suppose that $b \in \Lipb$, $\lambda \in C^3_b$ and $\psi \in C^2_b$. Then, for any $t,x$ and $\gamma$, we have that
\begin{equation}\label{eq:psiintbound}
\bigg|\int_t^T\psi(X^{t,x,\gamma}_s,\gamma_s)\,\rd\bmz_s\bigg| \leq C\Big(1 + \|\gamma\|_{\ptvartT}^{2(1 + p)}\Big),
\end{equation}
where the constant $C$ depends only on $b,\lambda,\psi,p,T$ and $L$.
\end{lemma}

\begin{proof}
By Theorem~\ref{thmsolnofRDE}, the RDE \eqref{eq:RDEcontroldynamics} has a unique solution $(X^{t,x,\gamma},\lambda(X^{t,x,\gamma},\gamma)) \in \sD^p_\zeta$ for any $\gamma \in \Cptvar$, and the integral $\int_t^\cdot\psi(X^{t,x,\gamma}_s,\gamma_s)\,\rd\bmz_s$ is then well-defined. By \eqref{eq:roughintbound}, we have that
\begin{align*}
\bigg|\int_t^T\psi(X^{t,x,\gamma}_s,\gamma_s)\,\rd\bmz_s\bigg| \leq C_p\Big(&|\psi(x,\gamma_t)\zeta_{t,T}| + \big\|R^{\psi(X,\gamma)}\big\|_{\ptvartT}\|\zeta\|_{\pvartT}\\
&+ \big|\psi(X^{t,x,\gamma},\gamma)'_t\zeta^{(2)}_{t,T}\big| + \big\|\psi(X,\gamma)'\big\|_{\pvartT}\big\|\zeta^{(2)}\big\|_{\ptvartT}\Big).
\end{align*}
Applying the estimates in Proposition~\ref{proproughestimates}, we deduce \eqref{eq:psiintbound}.
\end{proof}

\begin{remark}
The choice to measure the regularity of the controls using $q$-variation for $q = \frac{p}{2}$ ensures that $\frac{1}{p} + \frac{1}{q} > 1$, so that the corresponding Young integral of $\gamma$ against $\zeta$ always exists. It may be tempting to wonder whether the result of Lemma~\ref{lemmapsiintbound} could still hold using a bound based on the $q$-variation of controls for a larger value of $q$. However, this is not true in general. Indeed, if $\frac{1}{p} + \frac{1}{q} < 1$, then one can construct a sequence $\{(\zeta^n,\gamma^n)\}_{n \geq 1}$ of pairs of bounded variation paths such that $\|\zeta^n\|_p = 1 = \|\gamma^n\|_q$ for all $n \geq 1$, but such that $\int_0^T\gamma^n_s\,\rd\zeta^n_s \to \infty$ as $n \to \infty$, which would contradict \eqref{eq:psiintbound}.
\end{remark}

Preventing degeneracy of this control problem can essentially be thought of as preventing the size of the rough integral above from becoming arbitrarily large. Lemma~\ref{lemmapsiintbound} shows that one can control the size of this integral by the $\frac{p}{2}$-variation of the controls. However, since controls can exhibit arbitrarily large $\frac{p}{2}$-variation whilst remaining uniformly bounded, the cost functional in \eqref{eq:costfuncJ} is not able to adequately penalise this variation. In view of Example~\ref{naiveexample}, for a typical choice of $\psi$, one should expect the value function in \eqref{eq:valuefuncnaive} to be simply given by
$$v(t,x) = -\infty \qquad \text{for all} \quad (t,x) \in [0,T) \times \R^m.$$
We also point out that merely restricting the class of controls $\gamma$ to, say, smooth functions does nothing to resolve this problem.

The estimate in \eqref{eq:psiintbound} implies that one could prevent degeneracy by imposing a uniform bound on the $\frac{p}{2}$-variation of the controls but, as appreciated in Diehl et al.~\cite{DiehlFrizGassiat2017}, this would not be a very natural condition. Instead, we first propose to introduce an artificial cost in order to penalise this variation.

\begin{definition}\label{defnregulariser}
Let $\cS \subseteq \Cptvar$ be a Banach space of functions from $[0,T] \to \R^k$ (with a possibly stronger topology). We shall call a function $\beta \colon \simplex \times \Cptvar \to \R \cup \{+\infty\}$ a \emph{regularising cost on} $\cS$, if it is bounded below, takes the value $+\infty$ on $\simplex \times (\Cptvar \setminus \cS)$, and, for every $0 \leq r < t \leq T$, the map $\beta_{r,t} \colon \cS \to \R$ is continuous, and satisfies
\begin{equation}\label{eq:regulariser}
\frac{\beta_{r,t}(\gamma)}{\|\gamma\|_{\ptvarrt}^{2(1 + p)}}\, \longrightarrow\, \infty \qquad \text{as} \qquad \|\gamma\|_{\ptvarrt}\, \longrightarrow\, \infty.
\end{equation}
\end{definition}

An example of such a cost on $\Cptvar$ is given by
\begin{equation}\label{eq:exampleregularcost}
\beta_{r,t}(\gamma) = \varepsilon\|\gamma\|_{\ptvarrt}^q
\end{equation}
for any $\varepsilon > 0$ and $q > 2(1 + p)$.

\begin{remark}
We point out that the power in the denominator in \eqref{eq:regulariser} is sufficient but by no means necessary. This choice is a result of the estimate in \eqref{eq:psiintbound}, which we do not expect to be sharp.
\end{remark}

Instead of the naive value function in \eqref{eq:valuefuncnaive}, we consider the modified function given by\footnote{The restriction to controls $\gamma \in \gCptvar \subset \Cptvar$ is negligible. Indeed, we recall that $\Cqtvar \subset \gCptvar$ for any $q \in [2,p)$.}
\begin{equation}\label{eq:modifiedvaluefunc}
V(t,x) := \inf_{\gamma \in \gCptvar}\big\{J(t,x;\gamma) + \beta_{t,T}(\gamma)\big\},
\end{equation}
for some regularising cost $\beta$. In practice, the justification of the introduction of this `artificial cost' depends on the application one has in mind; we will see examples of this later in Sections~\ref{SecExamples} and \ref{SecFiltReformulation}.

The following proposition demonstrates the nondegeneracy of this modified control problem.

\begin{proposition}
Under the natural assumption that $f$ and $g$ are bounded below, the same is true of the value function $V$.
\end{proposition}

\begin{proof}
It follows from Lemma~\ref{lemmapsiintbound} and \eqref{eq:regulariser} that
$$\bigg|\int_t^T\psi(X^{t,x,\gamma}_s,\gamma_s)\,\rd\bmz_s\bigg| \leq C + \frac{\beta_{t,T}(\gamma)}{2}$$
for some new constant $C$, and hence that
$$J(t,x;\gamma) + \beta_{t,T}(\gamma) \geq \int_t^Tf(X^{t,x,\gamma}_s,\gamma_s)\,\rd s + g(X^{t,x,\gamma}_T) + \frac{\beta_{t,T}(\gamma)}{2} - C.$$
Since the cost functions $f,g$ and $\beta$ are all bounded below, the result follows.
\end{proof}

\subsection{Recovering dynamic programming}

We have seen that one can resolve the degeneracy of the optimal control problem by introducing an artificial cost to penalise the variation of the controls. In Definition~\ref{defnregulariser} we introduced a rather general class of cost functions which provide a sufficient penalisation. The problem with such cost functions, such as the one in \eqref{eq:exampleregularcost}, is that typically they are not additive, in the sense that $\beta_{r,s} + \beta_{s,t} \neq \beta_{r,t}$. A consequence of this is that the corresponding control problem is no longer dynamic. That is, the value function in \eqref{eq:modifiedvaluefunc} is not generally amenable to dynamic programming, and thus one cannot necessarily write down a PDE associated with the control problem. Our next aim will be to demonstrate the existence of an additive regularising cost on a more regular space of controls, which allows dynamic programming to be recovered.

\begin{lemma}
Let $\beta$ be a regularising cost on $\gCptvar$. Let $\cS \subseteq \gCptvar$ be a subset which contains all smooth functions from $[0,T] \to \R^k$. Then the value function defined in \eqref{eq:modifiedvaluefunc} satisfies
\begin{equation}\label{eq:restrictcontrols}
V(t,x) = \inf_{\gamma \in \cS}\big\{J(t,x;\gamma) + \beta_{t,T}(\gamma)\big\}.
\end{equation}
\end{lemma}

\begin{proof}
By definition, for any $\gamma \in \gCptvar$, there exists a sequence of smooth controls $\{\gamma^n\}_{n \geq 1}$ such that $\|\gamma^n - \gamma\|_\infty + \|\gamma^n - \gamma\|_{\frac{p}{2}} \to 0$ as $n \to \infty$. The result then follows from the continuity of $\beta_{t,T}$ and the stability estimates in Proposition~\ref{proproughstability}.
\end{proof}

In particular, \eqref{eq:restrictcontrols} holds with $\cS = W^{1,q}$ for any $q \geq 1$, where $W^{1,q} = W^{1,q}([0,T];\R^k)$ denotes the usual Sobolev space. We recall the continuous embeddings $W^{1,q} \hookrightarrow \cC^{1\text{-var}} \hookrightarrow \gCptvar$, exhibited by the inequalities
\begin{equation*}
T^{\frac{q - 1}{q}}\bigg(\int_r^t|\dot{\gamma}_s|^q\,\rd s\bigg)^{\hspace{-2pt}\frac{1}{q}} \geq \|\gamma\|_{\onevarrt} \geq \|\gamma\|_{\ptvarrt},
\end{equation*}
where we write $\dot{\gamma}$ for the unique element $\dot{\gamma} \in L^q([0,T];\R^k)$ such that $\rd\gamma_s = \dot{\gamma}_s\,\rd s$. It follows that, for any $\varepsilon > 0$ and $q > 2(1 + p)$, the choice
$$\beta_{r,t}(\gamma) = \varepsilon\int_r^t|\dot{\gamma}_s|^q\,\rd s$$
for $\gamma \in W^{1,q}$ (and $\beta(\gamma) \equiv \infty$ otherwise), defines a regularising cost on $W^{1,q}$. Moreover, $\beta$ is additive, in the sense that $\beta_{r,s} + \beta_{s,t} = \beta_{r,t}$ for all $r \leq s \leq t$; in other words, for each $\gamma \in W^{1,q}$, the two-parameter functional $\beta(\gamma) \colon \simplex \to \R$ is uniquely characterised by the path $t \mapsto \beta_{0,t}(\gamma)$. With this choice of $\beta$, we can now write
\begin{equation*}
V(t,x) = \inf_{a \in \R^k}v(t,x,a)
\end{equation*}
where, for $(t,x,a) \in [0,T] \times \R^m \times \R^k$,
\begin{equation}\label{eq:newvaluefunc}
v(t,x,a) := \inf_{u \in L^q}\bigg\{J(t,x;\gamma^{t,a,u}) + \varepsilon\int_t^T|u_s|^q\,\rd s\bigg\},
\end{equation}
with $\gamma^{t,a,u}_r := a + \int_t^ru_s\,\rd s$ for $r \in [t,T]$. The function $v$ is both nondegenerate, and satisfies the following:

\begin{proposition}[Dynamic programming principle]\label{propDPP}
Let us write $X^{t,x,a,u} := X^{t,x,\gamma^{t,a,u}}$. Then, for any $t,x,a$ and $r \in [t,T]$, with $v$ as in \eqref{eq:newvaluefunc}, we have
\begin{align*}
v(t,x,a) = \inf_{u \in L^q}\bigg\{&v(r,X^{t,x,a,u}_r,\gamma^{t,a,u}_r) + \int_t^rf(X^{t,x,a,u}_s,\gamma^{t,a,u}_s)\,\rd s\\
&\hspace{35pt} + \int_t^r\psi(X^{t,x,a,u}_s,\gamma^{t,a,u}_s)\,\rd\bmz_s + \varepsilon\int_t^r|u_s|^q\,\rd s\bigg\}.
\end{align*}
\end{proposition}

This result follows the same proof as that of Theorem~2.1 in \cite[Chapter~4]{YongZhou1999}. In particular, the rough integrals appearing in the controlled dynamics and value function do not cause any additional difficulty.

\subsection{A generalised dynamic control problem}

To summarise the previous subsections, we propose to reformulate the naive control problem, given originally by \eqref{eq:RDEcontroldynamics}--\eqref{eq:costfuncJ}, to resolve the degeneracy problem whilst retaining enough dynamic structure to retain dynamic programming, by restricting to a sufficiently regular space of controls, and introducing an additive artificial cost function, written in terms of the derivative of the controls. Rather than merely \eqref{eq:RDEcontroldynamics}, by including $\gamma$ as part of the state trajectory, we instead consider the controlled dynamics
\begin{align}
\rd X^{t,x,a,u}_s &= b(X^{t,x,a,u}_s,\gamma^{t,a,u}_s)\,\rd s + \lambda(X^{t,x,a,u}_s,\gamma^{t,a,u}_s)\,\rd\bmz_s, & X^{t,x,a,u}_t &= x,\label{eq:roughXdynamics}\\
\rd\gamma^{t,a,u}_s &= h(\gamma^{t,a,u}_s,u_s)\,\rd s, & \gamma^{t,a,u}_t &= a.\label{eq:controldynamics}
\end{align}
For generality, we have introduced the function $h \colon \R^k \times U \to \R^k$, where here $(U,\|\cdot\|_U)$ is a finite dimensional Banach space, and the control $u$ belongs to the space $\cU$ of bounded measurable functions $u \colon [0,T] \to U$.

We shall henceforth consider the cost functional
\begin{align*}
J(t,x,a;u) := \int_t^Tf&(X^{t,x,a,u}_s,\gamma^{t,a,u}_s,u_s)\,\rd s\\
&+ \int_t^T\psi(X^{t,x,a,u}_s,\gamma^{t,a,u}_s)\,\rd\bmz_s + g(X^{t,x,a,u}_T,\gamma^{t,a,u}_T)
\end{align*}
and the value function
\begin{equation}\label{eq:roughvaluefunc}
v(t,x,a) := \inf_{u \in \cU}J(t,x,a;u),
\end{equation}
where we have absorbed a regularising cost into the function $f \colon \R^m \times \R^k \times U \to \R$, which crucially is now also allowed to depend on $u$. There is also no harm in allowing the terminal cost $g \colon \R^m \times \R^k \to \R$ to depend on the terminal value of $\gamma$.

\begin{remark}
If the function $h = h(a,u)$ were bounded in $u$, then setting $\tilde{X} = (X,\gamma)$ would now put us into a comparable setting to Diehl et al.~\cite{DiehlFrizGassiat2017}. However, it is more natural here to allow $h$ to be unbounded in $u$, meaning that \cite[Theorem~5]{DiehlFrizGassiat2017} does not directly apply\footnote{This boundedness condition is not stated explicitly in \cite{DiehlFrizGassiat2017}, but is necessary for the application of \cite[Corollary~III.3.6]{BardiCapuzzoDolcetta2008} in the proof of \cite[Theorem~5]{DiehlFrizGassiat2017}; see Assumption (A\textsubscript{1}) in \cite[Chapter~III]{BardiCapuzzoDolcetta2008}.}. Moreover, in \cite{DiehlFrizGassiat2017} the cost functions $f$ and $g$ are assumed to be bounded, but we will relax this assumption in the current work.

The inclusion of the integral $\int\psi(X,\gamma)\,\rd\bmz$ in the value function also takes us outside the setting of \cite{DiehlFrizGassiat2017}. This term could be included in the terminal cost by setting $\tilde{X} = (X,\gamma,Z)$ and $\tilde{g}(x,a,z) = g(x,a) + z$ with $Z_r = z + \int_t^r\psi(X,\gamma)\,\rd\bmz$, albeit with the additional complication that the terminal cost $\tilde{g}$ would then be neither bounded from above nor below.
\end{remark}

\begin{assumption}\label{assumptionblambdamufgh}
We assume that
\begin{itemize}
\item $b \in \Lipb$ and $\lambda,\psi \in C^3_b$,
\item $f = f(x,a,u)$ and $g = g(x,a)$ are continuous, bounded below, and Lipschitz continuous in $(x,a)$, uniformly in $u$,
\item $h = h(a,u)$ is continuous, Lipschitz in $a$, uniformly in $u$, and is bounded in $a$, locally uniformly in $u$, and moreover, for some $\delta \geq 1$, satisfies
\begin{equation}\label{eq:hgrowth}
\sup_{a \in \R^k}\frac{\big|h(a,u)\big|}{\|u\|_U^\delta} \longrightarrow 0 \qquad \text{as} \quad\ \|u\|_U\, \longrightarrow\, \infty,
\end{equation}
\item with the same $\delta$ as in \eqref{eq:hgrowth}, the running cost $f$ satisfies
\begin{equation}\label{eq:fcoercivity}
\inf_{x \in \R^m,\, a \in \R^k}\frac{f(x,a,u)}{\|u\|_U^{2(1 + p)\delta}}\, \longrightarrow\, \infty \qquad \text{as} \quad\ \|u\|_U\, \longrightarrow\, \infty.
\end{equation}
\end{itemize}
\end{assumption}

\begin{remark}
One could in principle also allow the drift coefficient $b$ to depend on the control $u$. In this case it is less straightforward to obtain solutions to the RDE \eqref{eq:roughXdynamics}, but the necessary technical results have already been established in \cite{DiehlFrizGassiat2017}.
\end{remark}

The following lemma demonstrates the nondegeneracy of our newly formulated control problem.

\begin{lemma}\label{lemmapsiintboundCf}
For any $t,x,a$ and $u$, we have that
\begin{equation}\label{eq:psiboundCintf}
\bigg|\int_t^T\psi(X^{t,x,a,u}_s,\gamma^{t,a,u}_s)\,\rd\bmz_s\bigg| \leq C + \frac{1}{2}\int_t^Tf(X^{t,x,a,u}_s,\gamma^{t,a,u}_s,u_s)\,\rd s,
\end{equation}
where the constant $C$ depends only on $b,\lambda,\psi,h,p,T$ and $L$.
\end{lemma}

\begin{proof}
By \eqref{eq:psiintbound}, \eqref{eq:controldynamics} and H\"older's inequality, we have that
\begin{align*}
&\bigg|\int_t^T\psi(X^{t,x,a,u}_s,\gamma^{t,a,u}_s)\,\rd\bmz_s\bigg| \leq C\Big(1 + \|\gamma^{t,a,u}\|_{\ptvartT}^{2(1 + p)}\Big) \leq C\Big(1 + \|\gamma^{t,a,u}\|_{\onevartT}^{2(1 + p)}\Big)\\
&= C\bigg(1 + \bigg(\int_t^T\big|h(\gamma^{t,a,u}_s,u_s)\big|\,\rd s\bigg)^{\hspace{-2pt}2(1 + p)}\bigg) \leq C\bigg(1 + T^{\frac{2(1 + p)}{p'}}\int_t^T\big|h(\gamma^{t,a,u}_s,u_s)\big|^{2(1 + p)}\,\rd s\bigg)
\end{align*}
where $p'$ is the H\"older conjugate of $2(1 + p)$. Then, by \eqref{eq:hgrowth} and \eqref{eq:fcoercivity} (noting that, since $U$ is finite dimensional, $h$ is uniformly bounded on bounded subsets of $U$), we can ensure that \eqref{eq:psiboundCintf} holds for a new constant $C$.
\end{proof}

\begin{corollary}\label{corollaryrestrctrsoncpt}
Let $K$ be a compact subset of $\R^m \times \R^k$. There exists an $M > 0$ such that, when taking the infimum over $u \in \cU$ in \eqref{eq:roughvaluefunc} for $(t,x,a) \in [0,T] \times K$, one may restrict to controls $u$ satisfying $\|\gamma^{t,a,u}\|_{\frac{p}{2}} \leq M$.
\end{corollary}

\begin{proof}
By Lemma~\ref{lemmapsiintboundCf} and the assumption that $g$ is bounded below, we have that
$$J(t,x,a;u) \geq \frac{1}{2}\int_t^Tf(X^{t,x,a,u}_s,\gamma^{t,a,u}_s,u_s)\,\rd s - C$$
for some possibly new constant $C$. Let $u^\ast \in \cU$ be an arbitrary control. By the above, we may ignore all controls $u$ such that
$$\frac{1}{2}\int_t^Tf(X^{t,x,a,u}_s,\gamma^{t,a,u}_s,u_s)\,\rd s - C > \sup_{(\hat{t},\hat{x},\hat{a}) \in [0,T] \times K}J(\hat{t},\hat{x},\hat{a};u^\ast).$$
This gives an upper bound on $\int_t^Tf(X^{t,x,a,u}_s,\gamma^{t,a,u}_s,u_s)\,\rd s$, which we observe, by the proof of Lemma~\ref{lemmapsiintboundCf}, also implies an upper bound on $\|\gamma^{t,a,u}\|_{\frac{p}{2}}$.
\end{proof}

\subsection{A smooth noise approximation}\label{SecSmoothApprox}

Although a dynamic programming principle of the form in Proposition~\ref{propDPP} holds for the value function $v$ in \eqref{eq:roughvaluefunc}, the appearance of the rough integrals makes it less straightforward to derive a PDE directly from this result. As in Example~\ref{naiveexample}, we will therefore proceed by first approximating $\zeta$ by a smooth function $\eta$. We then define the corresponding approximate control problem, with dynamics
\begin{equation}
\rd X^{t,x,a,u,\eta}_s = b(X^{t,x,a,u,\eta}_s,\gamma^{t,a,u}_s)\,\rd s + \lambda(X^{t,x,a,u,\eta}_s,\gamma^{t,a,u}_s)\,\rd\eta_s, \qquad X^{t,x,a,u,\eta}_t = x,\label{eq:smoothXdynamics}
\end{equation}
where $\gamma^{t,a,u}$ satisfies \eqref{eq:controldynamics}. Naturally equation \eqref{eq:smoothXdynamics} has a unique $C^1$ solution. However, in the following it will be useful to also embed this solution in rough path space. As $\eta$ is smooth, we can simply enhance it with its iterated integrals in the classical Lebesgue--Stieltjes sense,
\begin{equation}\label{eq:smoothiteratedint}
\eta^{(2)}_{s,t} := \int_s^t\eta_{s,r} \otimes \rd\eta_r,
\end{equation}
so that $\bme = (\eta,\eta^{(2)})$ is itself a rough path. Any continuous path with finite $p$-variation would make a valid candidate for the Gubinelli derivative of $X^\eta$ (with respect to $\eta$), but to be consistent with the genuinely rough case above we insist on the choice $(X^\eta)' = \lambda(X^\eta,\gamma)$. We can then consider $(X^\eta,\lambda(X^\eta,\gamma))$ as the solution of \eqref{eq:smoothXdynamics} in the sense of Theorem~\ref{thmsolnofRDE}.

We also define the corresponding approximate value function $v^\eta$ as
\begin{align*}
v^\eta(t,x,a) := \inf_{u \in \cU}\bigg\{\int_t^T&f(X^{t,x,a,u,\eta}_s,\gamma^{t,a,u}_s,u_s)\,\rd s\\
&+ \int_t^T\psi(X^{t,x,a,u,\eta}_s,\gamma^{t,a,u}_s)\,\rd\eta_s + g(X^{t,x,a,u,\eta}_T,\gamma^{t,a,u}_T)\bigg\}.
\end{align*}

Writing $\deta$ for the derivative of $\eta$, under Assumption~\ref{assumptionblambdamufgh}, we can apply Theorem~3.2 in Bardi and Da Lio \cite{BardiDaLio1997}, to obtain that $v^\eta$ is the unique viscosity solution of the HJB equation
\begin{align}
-\frac{\pa v^\eta}{\pa t}(t,x,a) - b(x,a)\cdot\nabla_xv^\eta(t,x,a&) - \inf_{u \in U}\big\{h(a,u)\cdot\nabla_av^\eta(t,x,a) + f(x,a,u)\big\}\nonumber\\
&-\big(\lambda(x,a)\cdot\nabla_xv^\eta(t,x,a) + \psi(x,a)\big)\deta_t = 0,\label{eq:smoothHJB}
\end{align}
with the terminal condition
\begin{equation}\label{eq:terminalsmoothHJB}
v^\eta(T,x,a) = g(x,a).
\end{equation}
Moreover, by Theorem~2.2 in \cite{BardiDaLio1997}, this solution is locally Lipschitz continuous.

\subsection{A rough HJB equation}

Replacing $\eta$ in \eqref{eq:smoothHJB} with $\bmz$, we formally derive the rough PDE given by
\begin{equation}\label{eq:roughHJB}
-\rd v - b\cdot\nabla_xv\,\rd t - \inf_{u \in U}\big\{h\cdot\nabla_av + f\big\}\hspace{0.3pt}\rd t - \big(\lambda\cdot\nabla_xv + \psi\big)\hspace{0.3pt}\rd\bmz = 0,
\end{equation}
with
\begin{equation}\label{eq:terminalroughHJB}
v(T,x,a) = g(x,a).
\end{equation}
We point out that as written equation \eqref{eq:roughHJB} is only formal, and is given a precise meaning in Definition~\ref{defnsolnroughHJB} below.

The following definition exhibits a standard notion of solution for rough PDEs, used in \cite{DiehlFrizGassiat2017}, as well as for instance by Caruana, Friz and Oberhauser \cite{CaruanaFriz2009,CaruanaFrizOberhauser2011,FrizOberhauser2014} (see also Chapter~12 in \cite{FrizHairer2014}).

\begin{definition}\label{defnsolnroughHJB}
For any smooth function $\eta \colon [0,T] \to \R^d$, write $v^\eta$ for the unique viscosity solution of \eqref{eq:smoothHJB} and \eqref{eq:terminalsmoothHJB}. Moreover, write $\bme$ for the rough path obtained by enhancing $\eta$ with its iterated integrals in the Lebesgue--Stieltjes sense, as in \eqref{eq:smoothiteratedint}. We say that a continuous function $v$ solves \eqref{eq:roughHJB} and \eqref{eq:terminalroughHJB} if
$$v^{\eta^n} \longrightarrow\, v \qquad \text{as} \quad\ \ n\, \longrightarrow\, \infty$$
locally uniformly on $[0,T] \times \R^m \times \R^k$, whenever $(\eta^n)_{n \geq 1}$ is a sequence of smooth paths such that $\bme^n \to \bmz$ with respect to the $\frac{1}{p}$-H\"older rough path metric, i.e.~$\varrho_{\opHol}(\bme^n,\bmz) \to 0$ as $n \to \infty$.
\end{definition}

Note that uniqueness of such a solution is built into the definition. Moreover, note that since we assumed that $\bmz$ is a geometric rough path, there certainly exists such a sequence of smooth paths $(\eta^n)_{n \geq 1}$.

\begin{theorem}\label{thmvaluefuncsolvesroughHJB}
Under Assumption~\ref{assumptionblambdamufgh}, the value function $v$ defined in \eqref{eq:roughvaluefunc} solves \eqref{eq:roughHJB} and \eqref{eq:terminalroughHJB} in the sense of Definition~\ref{defnsolnroughHJB}. Moreover, writing $v = v^\zeta$, the map from $\sC^{0,p}_g([0,T];\R^d) \to \R$ given by $\bmz \mapsto v^\zeta(t,x,a)$ is locally uniformly continuous with respect to each of the rough path metrics $\varrho_p$ and $\varrho_{\opHol}$, locally uniformly in $(t,x,a)$.
\end{theorem}

\begin{proof}
Let $K$ be a compact subset of $\R^m \times \R^k$ and let $\bme \in \sC^p$ be another rough path such that $\varrho_{\opHol}(\bme,\bmz) \leq 1$. By possibly replacing $L$ by $L + 1$, we may assume that $\ver{\bme}_{\opHol} \leq L$. Let us write $X^\eta = X^{t,x,a,u,\eta}$ (resp.~$X^\zeta = X^{t,x,a,u,\zeta}$) for the solution of the RDE \eqref{eq:roughXdynamics} driven by $\bme$ (resp.~$\bmz$), and write $v^\eta$ (resp.~$v^\zeta$) for the corresponding value function, as defined in \eqref{eq:roughvaluefunc}.

By Corollary~\ref{corollaryrestrctrsoncpt}, there exists an $M > 0$ such that, for $(t,x,a) \in [0,T] \times K$, we may restrict to controls $u \in \cU^M \subseteq \cU$ satisfying $\|\gamma^{t,a,u}\|_{\frac{p}{2}} \leq M$, so that in particular the hypotheses of Proposition~\ref{proproughstability} are satisfied.

In the following we shall use $\lesssim$ to denote inequality up to a multiplicative constant which may depend on $b,\lambda,\psi,f,g,h,p,T,L$ and $M$. It follows from Proposition~\ref{proproughstability} that
\begin{equation*}
\|X^\eta - X^\zeta\|_\infty \lesssim \varrho_p(\bme,\bmz),
\end{equation*}
and
\begin{equation*}
\bigg\|\int_t^\cdot\psi(X^\eta_s,\gamma_s)\,\rd\bme_s - \int_t^\cdot\psi(X^\zeta_s,\gamma_s)\,\rd\bmz_s\bigg\|_\infty \lesssim \varrho_p(\bme,\bmz).
\end{equation*}
By the Lipschitz assumptions on $f$ and $g$, for any $(t,x,a) \in [0,T] \times K$, we have
\begin{align*}
\big|v^\eta&(t,x,a) - v^\zeta(t,x,a)\big|\\
&\leq \sup_{u \in \cU^M}\bigg|\int_t^T\big(f(X^{u,\eta}_s,\gamma^{u}_s,u_s) - f(X^{u,\zeta}_s,\gamma^{u}_s,u_s)\big)\,\rd s\\
&\qquad \qquad + \int_t^T\psi(X^{u,\eta}_s,\gamma^{u}_s)\,\rd\bme_s - \int_t^T\psi(X^{u,\zeta}_s,\gamma^{u}_s)\,\rd\bmz_s + g(X^{u,\eta}_T,\gamma^{u}_T) - g(X^{u,\zeta}_T,\gamma^{u}_T)\bigg|\\
&\lesssim \sup_{u \in \cU^M}\bigg(\int_t^T|X^{u,\eta}_s - X^{u,\zeta}_s|\,\rd s + \varrho_p(\bme,\bmz) + |X^{u,\eta}_T - X^{u,\zeta}_T|\bigg)\\
&\lesssim \varrho_p(\bme,\bmz) \lesssim \varrho_{\opHol}(\bme,\bmz).
\end{align*}
Taking a sequence of smooth paths $(\eta^n)_{n \geq 1}$ such that $\varrho_{\opHol}(\bme^n,\bmz) \to 0$, the required convergence follows by taking $\bme = \bme^n$ in the above. Since the approximate value functions $v^{\eta^n}$ are continuous, continuity of the function $v$ with respect to $(t,x,a)$ also follows from this convergence. The stated continuity of the value function with respect to the driving rough path is also immediate from the above.
\end{proof}

\begin{remark}
One could also introduce another Brownian motion $W$ and consider as controlled dynamics the hybrid It\^o-rough differential equation $\rd X = b(X,\gamma)\,\rd s + \sigma(X,\gamma)\,\rd W + \lambda(X,\gamma)\,\rd \bmz$. Just as in the classical case, the value function is then defined as the infimum (or supremum) over adapted controls of an expected cost function, and the associated HJB equation is then of second order; see Example~\ref{exampleinsider} below.
\end{remark}

\subsection{Examples}\label{SecExamples}

\begin{example}\label{exBMstratroughpath}
When a Brownian motion $W$ is enhanced with its iterated integrals in the sense of Stratonovich integration, i.e.
\begin{equation}
W^{(2)}_{s,t} := \int_s^tW_{s,r} \otimes \circ\hspace{1pt}\rd W_r,
\end{equation}
then, almost surely, $\mathbf{W} = (W,W^{(2)})$ defines a $\frac{1}{p}$-H\"older geometric rough path for any $p \in (2,3)$. The choice $\bmz = \mathbf{W}$ thus leads to the stochastic PDE
\begin{equation*}
-\rd v - b\cdot\nabla_xv\,\rd t - \inf_{u \in U}\big\{h\cdot\nabla_av + f\big\}\hspace{0.3pt}\rd t - \big(\lambda\cdot\nabla_xv + \psi\big)\hspace{-1.5pt}\circ\hspace{-1pt}\rd\mathbf{W} = 0.
\end{equation*}
\end{example}

\begin{example}[Insider trading revisited]\label{exampleinsider}
Let us return to the setting of Example~\ref{naiveexample}, where we recall that an agent is trading a stock with the benefit of some extra information not available to the rest of the market. We denote the agent's initial investment by $a$, and the rate at which they purchase new stock by $u$. The dynamics of the agent's wealth process $X$ and investment $\gamma$ are given by\footnote{When $\sigma \neq 0$ the inclusion of the Brownian motion takes us outside the class of problems considered above, but there is no conceptual change and we expect all of the analysis to follow with appropriate technical adjustments.}
\begin{align*}
\rd X^{t,x,a,u}_s &= \gamma^{t,a,u}_s(\sigma\,\rd W_s + \rd\zeta_s), & X^{t,x,a,u}_t &= x,\\
\rd \gamma^{t,a,u}_s &= u_s\,\rd s, & \gamma^{t,a,u}_t &= a,
\end{align*}
where $W$ is a Brownian motion and $\zeta$ is an arbitrary continuous path. Note that, since $\zeta$ is continuous and $\gamma^{t,a,u}$ is of finite variation, the integral $\int_t^\cdot\gamma^{t,a,u}_s\,\rd\zeta_s$ exists in the Riemann--Stieltjes sense, so there is no need here to lift $\zeta$ into rough path space.

Let us suppose that the agent must pay a transaction cost of $\varepsilon u^2$. The agent's expected terminal wealth is then given by the value function
$$v(t,x,a) = \sup_{u \in \cU}\E\bigg[X^{t,x,a,u}_T - \int_t^T\varepsilon u_s^2\,\rd s\bigg],$$
where $\cU$ is the space of progressively measurable $\R$-valued processes. In this case the HJB equation \eqref{eq:roughHJB} takes the form
\begin{equation}\label{eq:exampleroughHJB}
-\rd v - \frac{1}{2}a^2\sigma^2\frac{\pa^2v}{\pa x^2}\,\rd t - \frac{1}{4\varepsilon}\bigg(\frac{\pa v}{\pa a}\bigg)^{\hspace{-3pt}2}\rd t - a\frac{\pa v}{\pa x}\,\rd\zeta = 0
\end{equation}
with
\begin{equation}\label{eq:exampleroughterminal}
v(T,x,a) = x.
\end{equation}
Approximating $\zeta$ by a smooth function $\eta$, we obtain the classical HJB equation
\begin{equation*}
-\frac{\pa v^\eta}{\pa t} - \frac{1}{2}a^2\sigma^2\frac{\pa^2v^\eta}{\pa x^2} - \frac{1}{4\varepsilon}\bigg(\frac{\pa v^\eta}{\pa a}\bigg)^{\hspace{-3pt}2} - a\deta\frac{\pa v^\eta}{\pa x} = 0.
\end{equation*}
The solution of this equation along with the terminal condition \eqref{eq:exampleroughterminal} is given by
$$v^\eta(t,x,a) = x + (\eta_T - \eta_t)a + \frac{1}{4\varepsilon}\int_t^T(\eta_T - \eta_s)^2\,\rd s.$$
Recalling Definition~\ref{defnsolnroughHJB}, we obtain the solution of \eqref{eq:exampleroughHJB} and \eqref{eq:exampleroughterminal} as
$$v(t,x,a) = x + (\zeta_T - \zeta_t)a + \frac{1}{4\varepsilon}\int_t^T(\zeta_T - \zeta_s)^2\,\rd s.$$
Note that this quantity remains finite even when $\zeta$ is of infinite variation. Thus, an agent, even with perfect knowledge of the future stock price, subject to sufficient transaction costs, can only make a finite profit.
\end{example}

\section{Robust filtering}\label{SecApplicFilter}

\subsection{The Kalman--Bucy filter}

In this section we turn our attention to the problem of stochastic filtering under model uncertainty. Let us take an underlying filtered space $(\Omega,\cF,(\cF_t)_{t \geq 0})$. We suppose that an $\R^m$-valued signal process $S$ and an $\R^d$-valued observation process $Y$ satisfy the following pair of linear equations
\begin{align*}
\rd S_t &= \alpha_tS_t\,\rd t + \sigma_t\,\rd B^1_t,\\
\rd Y_t &= c_tS_t\,\rd t + \rd B^2_t,
\end{align*}
with the initial conditions $Y_0 = 0$ and $S_0 \sim N(\mu_0,\Sigma_0)$ for some $\mu_0 \in \R^m$ and $\Sigma_0 \in \Sm$, where $\Sm$ denotes the set of symmetric, positive definite $m \times m$-matrices. Here $B^1$ (resp.~$B^2$) is a standard $\R^l$(resp.~$\R^d$)-valued Brownian motion, and $\alpha \colon [0,T] \to \R^{m \times m}$, $\sigma \colon [0,T] \to \R^{m \times l}$ and $c \colon [0,T] \to \R^{d \times m}$ are parameters. Here we include the case when the signal noise and observation noise are correlated; we suppose that their quadratic covariation is given by
$$\rd\langle B^1,B^2\rangle_t = \rho_t\,\rd t,$$
for some correlation matrix $\rho \colon [0,T] \to \R^{l \times d}$. In the scalar case, the correlation should naturally satisfy $\rho^2 \leq 1$. The analogous assumption here is that the matrix $I - \rho\rho^\top$ be positive semi-definite, where $I$ denotes the $l \times l$ identity matrix.

We shall denote by $(\cY_t)_{t \geq 0}$ the (completed) natural filtration generated by $Y$. In short, the filtering problem is concerned with, at each time $t$, determining the best estimate for $S_t$ given $\cY_t$, that is, finding the best estimate for the current value of $S$, given our past observations of $Y$. The mathematical theory underpinning the filtering of stochastic systems is by now well understood; a particularly good exposition is given in Bain and Crisan \cite{BainCrisan2009}. As observed by Kalman and Bucy \cite{Kalman1960,KalmanBucy1961}, and subsequently studied by numerous authors in various contexts, in this setting where, crucially, the underlying dynamics are linear, the conditional distribution of $S_t$ given $\cY_t$ is Gaussian. Moreover, the conditional mean $q_t = \E[S_t\,|\,\cY_t]$ of this distribution satisfies the SDE
\begin{equation}\label{eq:meandynamics}
\rd q_t = \alpha_tq_t\,\rd t + (R_tc^\top_t\hspace{-1pt} + \sigma_t\rho_t)(\rd Y_t - c_tq_t\,\rd t),
\end{equation}
and the conditional variance $R_t = \E[(S_t - q_t)(S_t - q_t)^\top|\,\cY_t]$ satisfies the deterministic matrix Riccati equation
\begin{equation}\label{eq:variancedynamics}
\frac{\rd R_t}{\rd t} = \sigma_t\sigma^\top_t + \alpha_tR_t + R_t\alpha^\top_t - (R_tc^\top_t\hspace{-1pt} + \sigma_t\rho_t)(c_tR_t + \rho_t^\top\hspace{-2pt}\sigma^\top_t).
\end{equation}

The filtering equations above allow one to fully characterise the conditional distribution of the signal. However, this procedure assumes that we know a priori the exact values of the parameters $\alpha,\sigma,c$ and $\rho$. In practice these parameters must be estimated from data, and in adopting these estimates one concedes an additional source of statistical uncertainty. In the present work we are interested in incorporating this uncertainty directly into the construction of the filter. That is, we are interested in stochastic filtering for linear systems which is robust with respect to model uncertainty.

\subsection{Robust filtering via nonlinear expectations}\label{SecRobustFiltViaNonlinExp}

Robust filtering has been studied in various papers, predominantly in the engineering literature. A typical approach is to construct an optimization procedure based on a minimax estimator for the hidden state, whereby one attempts to minimize a maximum expected loss over the space of possible models. See for instance the work of Borisov \cite{Borisov2008,Borisov2011}, Miller and Pankov \cite{MillerPankov2005}, Siemenikhin, Lebedev and Platonov \cite{Siemenikhin2016,SiemenikhinLebedevPlatonov2005} or Verd\'u and Poor \cite{VerduPoor1984}. By design, such estimators take into account a generally large set of models, even though many of them should be considered to be very implausible, thus often sacrificing filter performance under the most statistically reasonable model. Another approach is that of $H_\infty$, as well as hybrid $H_2/H_\infty$ filtering, which examines the energy gain from the noise input to the filtering error and attempts to minimize this energy transfer subject to suitable constraints; see Aliyu and Boukas \cite{AliyuBoukas2009}, Chen and Zhou \cite{ChenZhou2002}, Khargonekar, Rotea and Baeyens \cite{KhargonekarRoteaBaeyens1996}, Xie, de Souza and Fu \cite{XiedeSouzaFu1991} or Yang and Ye \cite{YangYe2009}.

A new approach to filtering in the presence of uncertainty was introduced in \cite{Cohenuncertfilterdiscrete}, which utilises a nonlinear expectation described in terms of a penalty function, which describes how our uncertainty evolves through time. This penalty can be calculated recursively, and can be used to construct robust estimates for any number of nonlinear functionals of the signal process, as well as robust interval estimates analogous to classical confidence/credible intervals.

The first application of this approach in a continuous time setting was presented in \cite{AllanCohen2019}, which studies a similar setting the one described above. In that paper however, the parameter $c$ was assumed to be known, and the signal and observation noises were assumed to be uncorrelated. In the current work we shall relax these assumptions, and also allow a more general penalty, which in particular takes into account the statistical likelihood for different parameter choices. As we will see, this approach will lead to the derivation of a pathwise stochastic control problem, and thus require the central ideas of the previous sections in order to proceed.

We consider \emph{convex expectations}, that is maps $\cE(\,\cdot\,|\,\cY_t) \colon L^\infty(\cF) \to L^\infty(\cY_t)$ satisfying the properties of monotonicity, translation equivariance, normalization and convexity, which additionally satisfy the Fatou property. Equivalently, and more explicitly, we consider maps which admit a representation of the form
\begin{equation}\label{eq:nonlinearexprep}
\cE(\xi\,|\,\cY_t) = \esssup_{\Q \in \cQ_t}\big\{\E^\Q[\xi\,|\,\cY_t] - \beta(\Q\,|\,\cY_t)\big\},
\end{equation}
where $\cQ_t$ is a collection of equivalent probability measures, and $\beta(\hspace{1pt}\cdot\,|\,\cY_t)$ is a nonnegative $\cY_t$-measurable penalty function. See e.g.~F\"ollmer and Schied \cite{FollmerSchied2002,FollmerSchied2004} for a proper exposition of the theory of nonlinear expectations.

As can be inferred from \eqref{eq:nonlinearexprep}, in the context of model uncertainty, i.e.~uncertainty in the underlying probability measure, nonlinear expectations provide an evaluation of random variables which takes into account every admissible measure. In other words, they consider every plausible view of the world, and envisage the worst case scenario. However, in contrast with sublinear expectations, the inclusion of the penalty term $\beta(\hspace{1pt}\cdot\,|\,\cY_t)$ means that we can penalise different measures according to how unreasonable we consider them to be, thus restricting our attention to only those measures which we consider to be realistic. Convex expectations are in this sense less pessimistic than their sublinear counterparts.

In our setting, the class of admissible measures simply corresponds to the family of possible parameters $\alpha,\sigma,c,\rho,\mu_0$ and $\Sigma_0$ of the dynamics of the signal and observation processes. For notational brevity, we shall denote\footnote{One is not obliged to consider all of these parameters as being uncertain, but we will focus on this, the most general case.}
$$\gamma := (\alpha,\sigma,c,\rho),$$
and write
$$\Gamma := \R^{m \times m} \times \R^{m \times l} \times \R^{d \times m} \times \Upsilon$$
for the space in which $\gamma$ takes values, where $\Upsilon$ denotes the space of valid correlation matrices:
\begin{equation}\label{eq:defncorrel}
\Upsilon := \big\{\rho \in \R^{l \times d} : I - \rho\rho^\top \text{\ is positive definite}\big\} = \big\{\rho \in \R^{l \times d} : \lmax(\rho\rho^\top) < 1\big\},
\end{equation}
where $\lmax(\hspace{0.5pt}\cdot\hspace{0.5pt})$ denotes the largest eigenvalue. We write $\P^{\gamma,\mu_0,\Sigma_0}$ for the measure associated with the parameters $\gamma,\mu_0$ and $\Sigma_0$, and write $\E^{\gamma,\mu_0,\Sigma_0}$ for the corresponding expectation.

For a given uncertainty aversion parameter $k_1 > 0$ and exponent $k_2 \geq 1$, we define, for any real-valued bounded measurable function $\vp$, the convex expectation with the representation
\begin{equation}\label{eq:defnnonlinearexp}
\cE(\vp(S_t)\,|\,\cY_t) = \esssup_{\gamma,\mu_0,\Sigma_0}\bigg\{\E^{\gamma,\mu_0,\Sigma_0}[\vp(S_t)\,|\,\cY_t] - \bigg(\frac{1}{k_1}\beta(\gamma,\mu_0,\Sigma_0\,|\,\cY_t)\bigg)^{\hspace{-2pt}k_2}\bigg\}.
\end{equation}
Here the essential supremum is taken over all possible parameters $(\mu_0,\Sigma_0) \in \R^m \times \Sm$ for the initial distribution of the signal, and over all choices of parameters $\gamma$ governing the dynamics of $S$ and $Y$.

In view of the insights of the previous section, we anticipate the eventual need to restrict to a sufficiently regular space of parameters $\gamma$ (which we will later refer to as controls). We consequently make the following assumption.

\begin{assumption}
We shall take the space of possible parameters $\gamma$ to be the family of all absolutely continuous functions $\gamma \colon [0,T] \to \Gamma$ with bounded derivative.
\end{assumption}

The penalty function $\beta$ represents our opinion of how unreasonable different values of the parameters are. We shall discuss this term further in the next subsection. The uncertainty aversion parameters $k_1,k_2$ are included for generality, but will play no significant role in our analysis.

The nonlinear expectation defined above can be used to construct a `robust' point estimate of $\vp(S_t)$, as
$$\argmin_{\xi \in \R}\, \cE\big((\vp(S_t) - \xi)^2\,\big|\,\cY_t\big).$$
Moreover, the nonlinear expectation $\cE(\vp(S_t)\,|\,\cY_t)$ will typically overestimate the true value of $\vp(S_t)$, so one may therefore think of $\cE(\vp(S_t)\,|\,\cY_t)$ as an `upper' expectation. Defining the corresponding `lower' expectation by $-\cE(-\vp(S_t)\,|\,\cY_t)$, one can then construct a robust interval estimate for $\vp(S_t)$ via
$$\big[\hspace{-2pt}-\hspace{-1pt}\cE(-\vp(S_t)\,|\,\cY_t),\,\cE(\vp(S_t)\,|\,\cY_t)\big].$$

\subsection{The penalty function}

In \cite{AllanCohen2019} the penalty $\beta$ was assumed to be fixed a priori, i.e.~it only took our prior beliefs into account. Although the parameters of the underlying system are unknown, as we make new observations we may wish to use these observations to update our opinion of how reasonable different parameter choices are. We shall therefore suppose that this penalty takes the form of a negative log-posterior density. That is, we take
\begin{equation}\label{eq:betaneglogpiL}
\beta_t(\gamma,\mu_0,\Sigma_0\,|\,\cY_t) = -\log\hspace{-1pt}\Big(\pi_t(\gamma,\mu_0,\Sigma_0)L_t(\gamma,\mu_0,\Sigma_0\,|\,\cY_t)\Big),
\end{equation}
where $\pi$ and $L(\hspace{1pt}\cdot\,|\,\cY_t)$ denote the prior and likelihood respectively.

The penalty function in \eqref{eq:betaneglogpiL} is built from the log-likelihood function, a familiar object from classical statistics. Penalties based on log-likelihoods form the basis of the data-driven robust (DR) expectation of \cite{Cohen2017}, which allows the level of penalisation of different parameter choices to be recursively updated through time as we collect new observations. We refer to \cite{Cohen2017} for further discussion. Here we add to this an additional penalty based on our prior beliefs, which may be calibrated accordingly.

We shall assume that the prior takes the form
\begin{equation}\label{eq:priorpi}
-\log\pi_t(\gamma,\mu_0,\Sigma_0) = \int_0^t\frf(q_s,R_s,\gamma_s)\,\rd s + g(\mu_0,\Sigma_0),
\end{equation}
where the functions $\frf$ and $g$ may be calibrated to represent our prior beliefs about the plausibility of different parameter choices. Here $q$ and $R$ are the conditional mean and variance corresponding to the parameters $\gamma,\mu_0$ and $\Sigma_0$, given by the solutions of \eqref{eq:meandynamics} and \eqref{eq:variancedynamics}.

Note that the measures $\P^{\gamma,\mu_0,\Sigma_0}$ for different choices of $\gamma,\mu_0$ and $\Sigma_0$ are all equivalent on $\cY_t$. A natural choice for $L_t(\hspace{1pt}\cdot\,|\,\cY_t)$ is thus the Radon--Nikodym derivative
$$L_t(\gamma,\mu_0,\Sigma_0\,|\,\cY_t) = \bigg(\frac{\rd\P^{\gamma,\mu_0,\Sigma_0}}{\rd\P^{\gamma^\ast,\mu_0^\ast,\Sigma_0^\ast}}\bigg)_{\hspace{-3pt}\cY_t}$$
which is precisely the likelihood ratio of the (arbitrary) parameter choice $\gamma,\mu_0,\Sigma_0$, with respect to a (fixed) choice of reference parameters $\gamma^\ast,\mu_0^\ast,\Sigma_0^\ast$. We will now derive an explicit expression for this likelihood.

Recall (from e.g.~Bain and Crisan \cite[Chapter~2]{BainCrisan2009}) that for a given choice of parameters $\gamma,\mu_0,\Sigma_0$, the \emph{innovation} process $V$, given in this setting by
$$\rd V_s = \rd Y_s - c_sq_s\,\rd s,$$
is a $\cY_t$-adapted Brownian motion under $\P^{\gamma,\mu_0,\Sigma_0}$. Writing $q^\ast$ (resp.~$V^\ast$) for the conditional mean (resp.~innovation process) under the reference measure $\P^{\gamma^\ast,\mu_0^\ast,\Sigma_0^\ast}$,
we have
$$\rd V_s = \rd V^\ast_s - (c_sq_s - c^\ast_sq^\ast_s)\,\rd s.$$
Thus, by Girsanov's theorem (see e.g.~\cite[Chapter~15]{CohenElliott2015}), as $V$ and $V^\ast$ have the predictable representation property under their respective measures, we can represent the likelihood as a stochastic exponential, namely
$$L_t(\gamma,\mu_0,\Sigma_0\,|\,\cY_t) = \exp\bigg(\int_0^t(c_sq_s - c^\ast_sq^\ast_s)\cdot\rd V^\ast_s - \frac{1}{2}\int_0^t\big|c_sq_s - c^\ast_sq^\ast_s\big|^2\,\rd s\bigg).$$
Substituting $\rd V^\ast_s = \rd Y_s - c^\ast_sq^\ast_s\,\rd s$, a short calculation yields
$$-\log L_t(\gamma,\mu_0,\Sigma_0\,|\,\cY_t) = -\int_0^t(c_sq_s - c^\ast_sq^\ast_s)\cdot\rd Y_s + \frac{1}{2}\int_0^t\Big(|c_sq_s|^2 - |c^\ast_sq^\ast_s|^2\Big)\,\rd s.$$

Since the reference parameters are taken to be fixed, they simply amount to an additive constant in the above expression. That is,
\begin{equation}\label{eq:negloglikeliIto}
-\log L_t(\gamma,\mu_0,\Sigma_0\,|\,\cY_t) = -\int_0^tc_sq_s\cdot\rd Y_s + \frac{1}{2}\int_0^t|c_sq_s|^2\,\rd s + \text{const.}
\end{equation}
For simplicity we will henceforth omit this constant from our analysis, conceding that our penalty function is correct up to an additive constant. This constant may be reintroduced upon numerical computation of the nonlinear expectation, chosen to ensure that the penalty function always takes the value zero at its minimum.

It will be useful later to interpret the stochastic integral in \eqref{eq:negloglikeliIto} in the sense of Stratonovich, rather than that of It\^o. We therefore make the transformation
$$-\int_0^tc_sq_s\cdot\rd Y_s = -\int_0^tc_sq_s\circ\rd Y_s + \frac{1}{2}\big\langle cq,Y\big\rangle_t.$$
Recalling \eqref{eq:meandynamics}, and using the fact that $c$ is absolutely continuous and in particular of bounded variation, after some calculation we deduce that the quadratic covariation term is given by
$$\big\langle cq,Y\big\rangle_t = \int_0^t\tr\big(c_s(R_sc_s^\top + \sigma_s\rho_s)\big)\,\rd s,$$
where $\tr(\hspace{1pt}\cdot\hspace{1pt})$ denotes the trace. Note that $c_sR_sc_s^\top$ is positive semi-definite and therefore has nonnegative trace. Substituting back into \eqref{eq:negloglikeliIto}, we obtain
\begin{equation}\label{eq:negloglikeliStrat}
-\log L_t(\gamma,\mu_0,\Sigma_0\,|\,\cY_t) = -\int_0^tc_sq_s\circ\,\rd Y_s + \frac{1}{2}\int_0^t\Big(|c_sq_s|^2 + \tr\big(c_s(R_sc_s^\top + \sigma_s\rho_s)\big)\Big)\rd s.
\end{equation}
For notational consistency with Section~\ref{SecPathwiseControl}, we introduce the functions $f$ and $\psi$, given by
$$f(q,R,\gamma) := \frf(q,R,\gamma) + \frac{1}{2}\Big(|cq|^2 + \tr\big(c(Rc^\top + \sigma\rho)\big)\Big) \qquad \text{and} \qquad \psi(q,\gamma) := -cq,$$
where we recall $\gamma = (\alpha,\sigma,c,\rho)$. Combining \eqref{eq:betaneglogpiL}, \eqref{eq:priorpi} and \eqref{eq:negloglikeliStrat}, and substituting into \eqref{eq:defnnonlinearexp}, we then obtain the following representation.
\begin{align}
\cE(\vp(S_t)\,|\,\cY_t) = \esssup_{\gamma,\mu_0,\Sigma_0}\bigg\{&\E^{\gamma,\mu_0,\Sigma_0}[\vp(S_t)\,|\,\cY_t]\label{eq:nonlinexpStrat}\\
&\hspace{-21pt}- \bigg(\frac{1}{k_1}\bigg(\int_0^tf(q_s,R_s,\gamma_s)\,\rd s + \int_0^t\psi(q_s,\gamma_s)\circ\rd Y_s + g(\mu_0,\Sigma_0)\bigg)\bigg)^{\hspace{-2pt}k_2}\bigg\}.\nonumber
\end{align}

\subsection{Fixing an observation path}

Since the parameters $\alpha,\sigma,c$ and $\rho$ are assumed to be absolutely continuous, $R$ is then the $C^1$ solution of \eqref{eq:variancedynamics}, and it follows from integration by parts (see e.g.~Theorem~1.2.3 in \cite{Stroock2011}) that the It\^o integral against $Y$ in \eqref{eq:meandynamics} can also be interpreted pathwise as a Riemann--Stieltjes integral. Moreover, these two notions of integral coincide almost surely. This can be seen by noting that the corresponding Riemann sums converge almost surely to the Riemann--Stieltjes integral, but also in $L^2$ to the It\^o integral, so these two notions of integral must agree by the uniqueness of limits in probability.

In filtering we make inference based on observations of the process $Y$. Thus, it is natural to restrict our attention to a particular path of $Y$, which we denote by $\zeta$. That is, we define $\zeta \colon [0,T] \to \R^d$ by
\begin{equation*}
\zeta_s := Y_s(\omega) \qquad \text{for} \quad\ s \in [0,T],
\end{equation*}
for some fixed $\omega \in \Omega$. By the previous paragraph, we can then consider the filter dynamics \eqref{eq:meandynamics}--\eqref{eq:variancedynamics} with $Y$ replaced by $\zeta$, namely
\begin{gather}
\rd q_s = \alpha_sq_s\,\rd s + (R_sc^\top_s\hspace{-1pt} + \sigma_s\rho_s)(\rd\zeta_s - c_sq_s\,\rd s),\label{eq:qdynamics}\\
\frac{\rd R_s}{\rd s} = \sigma_s\sigma^\top_s + \alpha_sR_s + R_s\alpha^\top_s - (R_sc^\top_s\hspace{-1pt} + \sigma_s\rho_s)(c_sR_s + \rho_s^\top\hspace{-2pt}\sigma^\top_s),\label{eq:Rdynamics}
\end{gather}

\begin{remark}
Strictly speaking, we a priori only have that the solution $(q,R)$ of \eqref{eq:meandynamics}--\eqref{eq:variancedynamics} exists almost surely for each choice of parameters $\alpha,\sigma,c,\rho,\mu_0,\Sigma_0$. Here we actually wish to consider this solution for \emph{every} choice of parameters, for almost every \emph{fixed} $\omega \in \Omega$. This can be justified by first considering a countable dense collection of parameters, and then appealing to the stability of solutions to Lipschitz SDEs (see e.g.~Chapter~16 in \cite{CohenElliott2015}). Alternatively, having fixed an (arbitrary continuous) path $\zeta$, one can establish existence and uniqueness of solutions of $\eqref{eq:qdynamics}$ directly for any choice of parameters by a classical Picard iterative argument.
\end{remark}

Recall that the representation in \eqref{eq:nonlinexpStrat} for the nonlinear expectation involves the stochastic integral of $\psi(q,\gamma)$ against $Y$. Unlike the stochastic integral in \eqref{eq:meandynamics}, since the paths of $q$ and $Y$ both have Brownian-type regularity, in general this integral does not exist in the pathwise Riemann--Stieltjes sense. As in the previous section, we instead aim to interpret it as a rough integral. Similarly to the setting of Crisan et al.~\cite{CrisanDiehlFrizOberhauser2013}, this requires us to lift the observation process $Y$ into rough path space.

In the previous section we were able to solve optimal control problems where the driving noise was a geometric rough path. However, since It\^o integration does not satisfy first order calculus---that is, it does not satisfy the classical integration by parts/chain rule---when enhancements are defined using iterated It\^o integrals the resulting rough paths are in general not geometric. It was for this reason that we insisted on transforming the It\^o integral in \eqref{eq:negloglikeliIto} into the Stratonovich integral in \eqref{eq:negloglikeliStrat}. Similarly to Example~\ref{exBMstratroughpath}, by setting
$$Y^{(2)}_{s,t} = \int_s^tY_{s,r}\otimes\circ\,\rd Y_r,$$
we have that, almost surely, $\bY = (Y,Y^{(2)})$ defines a $\frac{1}{p}$-H\"older geometric rough path for any $p \in (2,3)$. Recalling that we defined $\zeta = Y(\omega)$ for a given $\omega \in \Omega$, we can now consider $\zeta$ as a rough path by defining its lift as
$$\bmz := \bY(\omega) \in \sC^{0,p}_g$$
for the same $\omega$.

It remains to establish $\psi(q,\gamma)$ as being controlled (in the sense of Gubinelli) by $\zeta$. The Gubinelli derivative of $q$ with respect to $\zeta$ can be inferred by simply inspecting \eqref{eq:qdynamics}. Indeed, recalling the notation $\zeta_{s,t} := \zeta_t - \zeta_s$, we have that
$$q_{s,t} = \int_s^t(R_rc_r^\top\hspace{-1pt} + \sigma_r\rho_r)\,\rd\zeta_r  + \text{O}\big(|t-s|\big) = (R_sc_s^\top\hspace{-1pt} + \sigma_s\rho_s)\zeta_{s,t} + \text{O}\big(|t-s|\big).$$
Since $c$ is of bounded variation, it is trivially controlled by $\zeta$ with derivative zero, and we conclude (from e.g.~Corollary~7.4 in \cite{FrizHairer2014}) that $\psi(q,\gamma) = -cq$ is indeed controlled by $\zeta$ with Gubinelli derivative $\psi(q,\gamma)' = -c(Rc^\top + \sigma\rho)$. Thus, almost surely,
$$\int_0^\cdot\psi(q_s,\gamma_s)\,\rd\bmz_s$$
exists as a rough integral and, moreover, coincides with the Stratonovich integral in \eqref{eq:nonlinexpStrat}.

\subsection{Reformulation as an optimal control problem}\label{SecFiltReformulation}

Writing $\gamma = (\alpha,\sigma,c,\rho)$ as usual, consider the functional $\kappa_t \colon \R^m \times \Sm \to \R$ defined by
\begin{equation}\label{eq:defnkappat}
\kappa_t(\mu,\Sigma) := \inf\bigg\{\int_0^tf(q_s,R_s,\gamma_s)\,\rd s + \int_0^t\psi(q_s,\gamma_s)\,\rd\bmz_s + g(q_0,R_0)\,\bigg|\, \genfrac{}{}{0pt}{}{\gamma,q_0,R_0 \hspace{8pt} \text{such that}}{(q_t,R_t) = (\mu,\Sigma)}\bigg\},
\end{equation}
where $q$ and $R$ satisfy \eqref{eq:qdynamics}--\eqref{eq:Rdynamics} with the \emph{terminal} condition $(q_t,R_t) = (\mu,\Sigma)$. The function $\kappa_t$ is related to the nonlinear expectation \eqref{eq:nonlinexpStrat} by the following lemma.

\begin{lemma}\label{lemmanonlinexpkappa}
Denote by $\Phi(\hspace{1pt}\cdot\,;\mu,\Sigma)$ the distribution function of a $N(\mu,\Sigma)$ distribution. For any rough path $\bmz = (\zeta,\zeta^{(2)}) = \bY(\omega) \in \sC^{0,p}_g$ as defined above, and any bounded measurable function $\vp$, we have the equality
\begin{equation}\label{eq:nonlinexpkappa}
\cE(\vp(S_t)\,|\,\cY_t) = \sup_{(\mu,\Sigma) \in \R^m \times \Sm}\bigg\{\int_{\R^m}\vp(x)\,\rd\Phi(x;\mu,\Sigma) - \bigg(\frac{1}{k_1}\kappa_t(\mu,\Sigma)\bigg)^{\hspace{-3pt}k_2}\bigg\},
\end{equation}
where the expectation on the left-hand side is evaluated on the realisation $Y = \zeta$.
\end{lemma}

The proof of Lemma~\ref{lemmanonlinexpkappa} is the same as that of Proposition~2.1 in \cite{AllanCohen2019}.

The expression for $\kappa$ in \eqref{eq:defnkappat} looks very much like that of the value function of an optimal control problem with state trajectories governed by \eqref{eq:qdynamics}--\eqref{eq:Rdynamics}. To make this exact we should write $\kappa$ as an infimum over the `control' $\gamma$ alone. This is easy, but one should note that, for certain choices of control $\gamma$ and terminal condition $(\mu,\Sigma) \in \R^m \times \Sm$, there will not actually exist a corresponding initial value $(q_0,R_0) \in \R^m \times \Sm$. This can happen for one of two reasons. First, due to the term $\sigma_s\sigma^\top_s$, the solution to \eqref{eq:Rdynamics} may no longer be positive semi-definite, so that $R_0$ does not correspond to a covariance matrix. Second, the solution to \eqref{eq:Rdynamics} may `blow up' in finite time, due to the quadratic term (in $R$) in the final term on the right-hand side of \eqref{eq:Rdynamics}. An example of such behaviour is exhibited in \cite[Section~3]{AllanCohen2019}.

Heuristically, the Kalman--Bucy filter is well behaved when run forwards in time from an initial condition, but here we instead fix a terminal condition and run the filtering equations backwards in time, which introduces the abnormalities described above. To prevent this unphysical behaviour we simply prescribe the value $g(\mu_0,\Sigma_0) = \infty$ for any initial value $(\mu_0,\Sigma_0) \notin \R^m \times \Sm$ and, although we don't actually obtain a physical initial value for solutions which `blow up' in a finite time, we assign an infinite `initial' cost to all such trajectories.

We can now write
\begin{equation}\label{eq:kappatinfovergamma}
\kappa_t(\mu,\Sigma) = \inf_{\gamma}\bigg\{\int_0^tf(q^{t,\mu,\Sigma,\gamma}_s,R^{t,\Sigma,\gamma}_s,\gamma_s)\,\rd s + \int_0^t\psi(q^{t,\mu,\Sigma,\gamma}_s,\gamma_s)\,\rd\bmz_s + g\big(q^{t,\mu,\Sigma,\gamma}_0,R^{t,\Sigma,\gamma}_0\big)\bigg\}
\end{equation}
where $q^{t,\mu,\Sigma,\gamma},R^{t,\Sigma,\gamma}$ satisfy \eqref{eq:qdynamics}--\eqref{eq:Rdynamics} with the terminal condition
$$\big(q^{t,\mu,\Sigma,\gamma}_t,R^{t,\Sigma,\gamma}_t\big) = (\mu,\Sigma),$$
noting that trajectories with the undesired behaviour described above will never be considered when taking the infimum in \eqref{eq:kappatinfovergamma}.

We have derived an optimal control problem, with the controlled dynamics \eqref{eq:qdynamics}--\eqref{eq:Rdynamics}, and the value function defined in \eqref{eq:kappatinfovergamma}. Moreover, the appearance of the `Brownian-like' path $\zeta$ in \eqref{eq:qdynamics}, and indeed the rough path $\bmz$ in \eqref{eq:kappatinfovergamma}, puts us back into the setting of pathwise stochastic control. In the case where the parameter $c$ is known, the signal and observation noises are uncorrelated (so that $\rho \equiv 0$), and if we omit the likelihood term in the penalty of our nonlinear expectation, then we are not directly controlling the coefficient of the rough term $\zeta$. This was the case in the setting of \cite{AllanCohen2019}, where a change of variables was then used to completely isolate the observation path from the controlled terms.

In the current setting however we cannot escape the need to control the coefficient of $\zeta$. As described in Section~\ref{SecPathwiseControl}, if the variation of the controls $\gamma$ is not sufficiently penalised then the control problem degenerates. The physical interpretation here is the following: even if we suppose that the parameters $\alpha,\sigma,c$ and $\rho$ are able to fluctuate at the same rate as the observation path $\zeta$, it is not reasonable to suppose that we should be able to calibrate these parameters over time scales that are so small that our observations are dominated by measurement noise.

Accordingly, we employ the strategy introduced in the previous section of introducing a regularising cost, and rewriting the problem in terms of an abstract control process $u$. We consider the dynamics \eqref{eq:qdynamics}--\eqref{eq:Rdynamics} along with
\begin{equation*}
\rd\gamma^{t,a,u}_s = h(\gamma^{t,a,u}_s,u_s)\,\rd s,
\end{equation*}
for some function $h \colon \Gamma \times U \to U$, where $u$ belongs to the class $\cU$ of bounded measurable functions $u \colon [0,T] \to U := \R^{m \times m} \times \R^{m \times l} \times \R^{d \times m} \times \R^{l \times d}$. The terminal condition is now given by
\begin{equation*}
\big(q^{t,\mu,\Sigma,a,u}_t,R^{t,\Sigma,a,u}_t,\gamma^{t,a,u}_t\big) = (\mu,\Sigma,a) \in \R^m \times \Sm \times \Gamma.
\end{equation*}
Allowing $f$ to depend on $u$, and $g$ to depend on $\gamma_0$ (which makes no difference to the proof of Lemma~\ref{lemmanonlinexpkappa}), and writing $\tilde{\kappa}$ for the regularised version of $\kappa$, we can write
\begin{equation}\label{eq:kappainfv}
\tilde{\kappa}_t(\mu,\Sigma) = \inf_{a \in \Gamma}\hspace{1pt}v(t,\mu,\Sigma,a),
\end{equation}
where
\begin{align}
v(t,\mu,\Sigma,a) := \inf_{u \in \cU}\bigg\{\int_0^t&f(q^{t,\mu,\Sigma,a,u}_s,R^{t,\Sigma,a,u}_s,\gamma^{t,a,u}_s,u_s)\,\rd s\label{eq:defnvrobust}\\
&+ \int_0^t\psi(q^{t,\mu,\Sigma,a,u}_s,\gamma^{t,a,u}_s)\,\rd\bmz_s + g\big(q^{t,\mu,\Sigma,a,u}_0,R^{t,\Sigma,a,u}_0,\gamma^{t,a,u}_0\big)\bigg\}\nonumber
\end{align}
is the value function of our new control problem. As before, to avoid unphysical trajectories, we assign an infinite cost to any controls $u$ such that $(\mu_0,\Sigma_0) \notin \R^m \times \Sm$ or such that the solution to \eqref{eq:Rdynamics} `blows up' in a finite time. Moreover, we assign an infinite cost to those controls which lead to $\rho$ leaving the space of valid correlation matrices $\Upsilon$ (as defined in \eqref{eq:defncorrel}).

Our uncertainty is thus represented by the function $v$. Once the value of this function has been determined, one can use \eqref{eq:kappainfv} and then \eqref{eq:nonlinexpkappa} to evaluate arbitrary functions of the signal process $S$ under the nonlinear expectation $\cE(\hspace{1pt}\cdot\hspace{2pt}|\cY_t)$.

\subsection{A nonlinear backward control problem}

It remains to characterise the value function $v$ defined in \eqref{eq:defnvrobust} as the unique solution of a rough HJB equation. For convenience, we rewrite the controlled dynamics in full as
\begin{align*}
\rd q^{t,\mu,\Sigma,a,u}_s &= b_\mu(q^{t,\mu,\Sigma,a,u}_s,R^{t,\Sigma,a,u}_s,\gamma^{t,a,u}_s)\,\rd s + \lambda(R^{t,\Sigma,a,u}_s,\gamma^{t,a,u}_s)\,\rd\zeta_s, & q^{t,\mu,\Sigma,a,u}_t &= \mu,\\
\rd R^{t,\Sigma,a,u}_s &= b_\Sigma(R^{t,\Sigma,a,u}_s,\gamma^{t,a,u}_s)\,\rd s, & R^{t,\Sigma,a,u}_t &= \Sigma,\\
\rd\gamma^{t,a,u}_s &= h(\gamma^{t,a,u}_s,u_s)\,\rd s, & \gamma^{t,a,u}_t &= a,
\end{align*}
where $\gamma = (\alpha,\sigma,c,\rho)$, and we define
\begin{align*}
b_\mu(q,R,\gamma) &= \alpha q - (Rc^\top + \sigma\rho)cq,\\
b_\Sigma(R,\gamma) &= \sigma\sigma^\top + \alpha R + R\alpha^\top - (Rc^\top + \sigma\rho)(cR + \rho^\top\hspace{-2pt}\sigma^\top),\\
\lambda(R,\gamma) &= Rc^\top + \sigma\rho.
\end{align*}

We note that this is a `backward' control problem in the sense that, in contrast to the classical setting of optimal control, here we prescribe a terminal condition for the state trajectories, and consider a cost associated with their initial value. More significantly, we note that $b_\mu,b_\Sigma \notin \Lipb$ and $\lambda,\psi \notin C^3_b$, so we cannot immediately apply the results of the previous section. Nevertheless, as we will see, the desired results can be recovered with some modifications.

\begin{notation}
In the following we write $|\cdot|$ for the usual Euclidean norm, and $\|A\|$ for the Frobenius norm of a given matrix $A$, i.e.~$\|A\|^2 = \tr(A^\top\hspace{-1pt}A)$. Given an element $\gamma = (\alpha,\sigma,c,\rho)$ of $U = \R^{m \times m} \times \R^{m \times l} \times \R^{d \times m} \times \R^{l \times d}$, we write $\|\gamma\| = \max\big\{\|\alpha\|,\|\sigma\|,\|c\|,\|\rho\|\big\}$. We point out however that, since the space $\Upsilon$ of correlation matrices is uniformly bounded\footnote{Indeed, one can show that $\|\rho\| \leq \sqrt{l}$ for every $\rho \in \Upsilon$.}, the dependence on $\|\rho\|$ is not particularly crucial.

If $A \in \Sm$, so that in particular $A$ is symmetric and positive definite, we write $\lmin(A)$ (resp.~$\lmax(A)$) for the smallest (resp.~largest) eigenvalue of $A$.

Where there is no risk of ambiguity, we will omit the superscripts from the state variables $q,R$ and $\gamma$. Finally, in this section we will use the symbol $\lesssim$ to denote inequality up to a multiplicative constant which may depend on any of the dimensions $d,l,m$, the functions $f,g,h$, the measure of regularity $p$, the terminal time $T$, and the bound $L$, where as usual $L > 0$ is chosen such that $\ver{\bmz}_{\opHol} \leq L$.
\end{notation}

\begin{assumption}\label{filtassumptions}
We assume that
\begin{itemize}
\item $f = f(q,R,\gamma,u)$ and $g = g(q,R,\gamma)$ are continuous, bounded below, and locally Lipschitz in $(q,R,\gamma)$, uniformly in $u$,
\item $h = h(\gamma,u)$ is continuous, surjective in $u$, i.e.~$\{h(\gamma,u) : u \in U\} = U$ for every $\gamma \in \Gamma$, Lipschitz in $\gamma$, uniformly in $u$, and bounded in $\gamma$, locally uniformly in $u$, and moreover, for some $\delta_1 \geq 1$, satisfies
\begin{equation}\label{eq:filthgrowth}
\sup_{\gamma \in \Gamma}\frac{\big\|h(\gamma,u)\big\|}{\|u\|^{\delta_1}}\, \longrightarrow\, 0 \qquad \text{as} \quad \|u\|\, \longrightarrow\, \infty,
\end{equation}
\item for some $\delta_2 > \delta_1$, the running cost $f$ satisfies the asymptotic growth condition:
\begin{equation}\label{eq:fgrowth}
\frac{f(q,R,\gamma,u)}{\big(1 + |q| + \|R\|^2 + \|\gamma\|^2\big)\|u\|^{\delta_2} + \big(1 + |q|^2 + \|R\|^2\big)\big(1 + \|\gamma\|^4\big)}\, \longrightarrow\, \infty
\end{equation}
as $|q| + \|R\| + \|\gamma\| + \|u\| \to \infty$,
\item and the initial cost $g$ satisfies:
\begin{align}
\frac{g(q,R,\gamma)}{|q|^2 + \big(1 + \|R\|\big)\big(1 + \|\gamma\|^2\big)}\, \longrightarrow\, \infty \qquad &\text{as} \qquad |q| + \|R\| + \|\gamma\|\, \longrightarrow\, \infty,\label{eq:ggrowth}\\
\inf_{(q,\gamma) \in \R^m \times \Gamma}g(q,R,\gamma)\, \longrightarrow\, \infty \qquad &\text{as} \qquad \lmin(R)\, \longrightarrow\, 0.\label{eq:ggrowtheigR}
\end{align}
\end{itemize}
If the correlation $\rho$ is known, then one can simply take, for example, $h(\gamma,u) = u$. If $\rho$ is uncertain then, for mostly technical reasons, one must take a little extra care to ensure that correlations close to the boundary of $\Upsilon$ (where $\lmax(\rho\rho^\top) = 1$) are sufficiently penalised. In this case we assume in addition that
\begin{itemize}
\item \begin{equation}\label{eq:ggrowthinrho}
\inf_{q,R,\alpha,\sigma,c}g(q,R,\gamma) \,\longrightarrow\, \infty \qquad \text{as} \quad \lmax(\rho\rho^\top) \,\longrightarrow\, 1,
\end{equation}
\begin{equation}\label{eq:hboundinrho}
\big\|h(\gamma,u)\big\| \leq (1 - \lmax(\rho\rho^\top))\|u\| \qquad \text{for all} \quad (\gamma,u) \in \Gamma \times U.
\end{equation}
\end{itemize}
\end{assumption}

\begin{remark}
The surjectivity of $h$ in $u$ is assumed to ensure that, no matter the choice of terminal condition $(t,\mu,\Sigma,a)$, there always exists a choice of control $u$ such that the state trajectories remain inside their respective domains, so that in particular $R_0 \in \Sm$ and $\rho_0 \in \Upsilon$. This guarantees that the value function $v$ is finite-valued.
\end{remark}

The result of Lemma~\ref{lemmapsiintboundCf} can be recovered in the current setting as follows.

\begin{lemma}\label{lemmafiltintpsibound}
Under Assumption~\ref{filtassumptions}, for any terminal condition $(t,\mu,\Sigma,a)$ and control $u$, we have that
\begin{equation*}
\bigg|\int_0^t\psi(q_s,\gamma_s)\,\rd\bmz_s\bigg| \leq C + \frac{1}{2}\bigg(\int_0^tf(q_s,R_s,\gamma_s,u_s)\,\rd s + g(q_0,R_0,\gamma_0)\bigg),
\end{equation*}
where the constant $C$ depends on $d, l, m, f, g, h, p, T$ and $L$.
\end{lemma}

\begin{proof}
From the estimate \eqref{eq:roughintbound} in Proposition~\ref{proproughintegral}, we have that
\begin{equation}\label{eq:intpsidzetabound}
\bigg|\int_0^t\psi(q_s,\gamma_s)\,\rd\bmz_s\bigg| \lesssim \big|\psi(q_0,\gamma_0)\big| + \big\|\psi(q,\gamma)'_0\big\| + \big\|R^{\psi(q,\gamma)}\big\|_{\ptvarzt} + \big\|\psi(q,\gamma)'\big\|_{\pvarzt}.
\end{equation}
We aim to bound each of the terms on the right-hand side. Recalling that $\psi(q,\gamma) = -cq$ and $\psi(q,\gamma)' = -c(Rc^\top + \sigma\rho)$, we have that
\begin{align}
\big|\psi(q_0,\gamma_0)\big| + \big\|\psi(q,\gamma)'_0\big\| &\lesssim \|c_0\||q_0| + \|c_0\|\big\|R_0c^\top_0\hspace{-1pt} + \sigma_0\rho_0\big\|\nonumber\\
&\lesssim |q_0|^2 + \big(1 + \|R_0\|\big)\|\gamma_0\|^2.\label{eq:psipsidashtbound}
\end{align}
Writing $(\dot{\alpha},\dot{\sigma},\dot{c},\dot{\rho}) = \dot{\gamma} = h(\gamma,u)$, we have
\begin{align}
\big\|\psi(q,\gamma)'\big\|_{\pvarzt} &\leq \big\|\psi(q,\gamma)'\big\|_{\onevarzt} = \big\|c(Rc^\top + \sigma\rho)\big\|_{\onevarzt}\nonumber\\
&= \int_0^t\Big|\dot{c}_s\big(R_sc^\top_s + \sigma_s\rho_s\big) + c_s\big(b_\Sigma(R_s,\gamma_s)c^\top_s + R_s\dot{c}^\top_s + \dot{\sigma}_s\rho_s + \sigma_s\dot{\rho}_s\big)\Big|\,\rd s\nonumber\\
&\lesssim \int_0^t\big(1 + \|R_s\|^2 + \|\gamma_s\|^2\big)\big\|h(\gamma_s,u_s)\big\| + \big(1 + \|R_s\|^2\big)\big(1 + \|\gamma_s\|^4\big)\,\rd s.\label{eq:psidashbound}
\end{align}
By the Young--L\'oeve inequality (see e.g.~Theorem~6.8 in \cite{FrizVictoir2010}), we have
\begin{equation}\label{eq:YoungLove}
\bigg|\int_h^r(R_sc^\top_s + \sigma_s\rho_s)\,\rd\zeta_s - (R_hc^\top_h + \sigma_h\rho_h)\zeta_{h,r}\bigg| \leq \frac{1}{1 - 2^{-\frac{1}{p}}}\big\|Rc^\top + \sigma\rho\big\|_{\onevarhr}\|\zeta\|_{\pvarhr}
\end{equation}
for any interval $[h,r] \subset [0,t]$. We calculate
\begin{align*}
-R^{\psi(q,\gamma)}_{h,r} &= -\psi(q_r,\gamma_r) + \psi(q_h,\gamma_h) + \psi(q,\gamma)'_h\zeta_{h,r}\\
&= c_hq_{h,r} + c_{h,r}q_r + \psi(q,\gamma)'_h\zeta_{h,r}\\
&= c_h\bigg(\int_h^r \rd q_s - (R_hc^\top_h\hspace{-1pt} + \sigma_h\rho_h)\zeta_{h,r}\bigg) + \int_h^r\dot{c}_sq_r\,\rd s.
\end{align*}
Recalling \eqref{eq:qdynamics} and using \eqref{eq:YoungLove}, we have
\begin{align*}
\big|R^{\psi(q,\gamma)}_{h,r}\big| &\lesssim \|c_h\|\big\|Rc^\top\hspace{-1pt} + \sigma\rho\big\|_{\onevarhr} + \bigg|\int_h^r c_h\big(\alpha_sq_s - (R_sc^\top_s\hspace{-1pt} + \sigma_s\rho_s)c_sq_s\big) + \dot{c}_sq_r\,\rd s\bigg|\\
&\lesssim \int_h^r\Big(\|c_h\|\big|b_\Sigma(R_s,\gamma_s)c^\top_s + R_s\dot{c}^\top_s + \dot{\sigma}_s\rho_s + \sigma_s\dot{\rho}_s\big|\\
&\hspace{50pt} + \|c_h\|\big|\alpha_sq_s - (R_sc^\top_s\hspace{-1pt} + \sigma_s\rho_s)c_sq_s\big| + \|\dot{c}_s\||q_r|\Big)\,\rd s.
\end{align*}
We then obtain
\begin{align}
\big\|&R^{\psi(q,\gamma)}\big\|_{\ptvarzt} \leq \big\|R^{\psi(q,\gamma)}\big\|_{\onevarzt} = \lim_{|\cP| \to 0}\sum_{[h,r] \in \cP}\big|R^{\psi(q,\gamma)}_{h,r}\big|\nonumber\\
&\lesssim \int_0^t\Big(\|c_s\|\big|b_\Sigma(R_s,\gamma_s)c^\top_s + R_s\dot{c}^\top_s + \dot{\sigma}_s\rho_s + \sigma_s\dot{\rho}_s\big|\nonumber\\
&\hspace{50pt} + \|c_s\|\big|\alpha_sq_s - (R_sc^\top_s\hspace{-1pt} + \sigma_s\rho_s)c_sq_s\big| + \|\dot{c}_s\||q_s|\Big)\,\rd s\nonumber\\
&\lesssim \int_0^t\big(1 + |q_s| + \|R_s\|^2 + \|\gamma_s\|^2\big)\big\|h(\gamma_s,u_s)\big\| + \big(1 + |q_s|^2 + \|R_s\|^2\big)\big(1 + \|\gamma_s\|^4\big)\,\rd s,\label{eq:Rpsibound}
\end{align}
where the limit in the above is taken over any sequence of partitions of the interval $[0,t]$ with mesh size tending to zero. Substituting \eqref{eq:psipsidashtbound}, \eqref{eq:psidashbound} and \eqref{eq:Rpsibound} into \eqref{eq:intpsidzetabound}, and using the growth conditions \eqref{eq:filthgrowth}--\eqref{eq:ggrowth} in Assumption~\ref{filtassumptions}, we deduce the result.
\end{proof}

\begin{corollary}\label{corofiltrestrctrsoncpt}
Let $K$ be a compact subset of $\R^m \times \Sm \times \Gamma$. There exists an $M > 0$ such that, when taking the infimum over $u \in \cU$ in \eqref{eq:defnvrobust} for $(t,\mu,\Sigma,a) \in [0,T] \times K$, one may restrict to controls $u$ such that the norms
$$\|q\|_\infty,\ \|R\|_\infty,\ \|\gamma\|_\infty,\ \|R\|_{\onevarzt},\ \|\gamma\|_{\onevarzt}$$
are all bounded by $M$.
\end{corollary}

\begin{proof}
One can obtain a bound for $\|\gamma\|_{\onevarzt}$ by a similar argument to that in the proof of Corollary~\ref{corollaryrestrctrsoncpt}. Since $\gamma_t = a$ lives in a compact set, we immediately also have a bound for $\|\gamma\|_\infty$.

We infer from \eqref{eq:ggrowth} that both the terminal value $R_t = \Sigma$ and initial value $R_0$ of $R$ must lie in some bounded set, and by inspecting the ODE \eqref{eq:Rdynamics} satisfied by $R$, we deduce that the entire path of $R$ must also live in a bounded set, giving a bound for $\|R\|_\infty$.

Given the bounds for $\|\gamma\|_\infty$ and $\|R\|_\infty$, a bound for $\|R\|_{\onevarzt}$ follows easily from \eqref{eq:Rdynamics}. Finally, inspecting the equation \eqref{eq:qdynamics} satisfied by $q$, in view of \eqref{eq:YoungLove}, we deduce a bound for $\|q\|_\infty$.
\end{proof}

As in the previous sections, let us approximate the rough path $\bmz$ by a smooth path $\eta$. We then obtain the approximate value function
\begin{align}
v^\eta(t,\mu,\Sigma,a) := \inf_{u \in \cU}\bigg\{\int_0^t&f(q^{t,\mu,\Sigma,a,u,\eta}_s,R^{t,\Sigma,a,u}_s,\gamma^{t,a,u}_s,u_s)\,\rd s\label{eq:defnvrobustsmooth}\\
&+ \int_0^t\psi(q^{t,\mu,\Sigma,a,u,\eta}_s,\gamma^{t,a,u}_s)\,\rd\eta_s + g\big(q^{t,\mu,\Sigma,a,u,\eta}_0,R^{t,\Sigma,a,u}_0,\gamma^{t,a,u}_0\big)\bigg\}.\nonumber
\end{align}

\begin{notation}
In the following, we will write $\nabla_\mu$ for the usual gradient with respect to $\mu$, and write $\nabla_\Sigma$ and $\nabla_a$ for the generalised gradients with respect to each of the components of $\Sigma$ and $a$ respectively. We will also write $A:B$ for the inner product of two elements $A,B$ from the same vector space. In particular, when $A,B$ are matrices, $A:B = \tr(A^\top\hspace{-1pt}B)$ denotes the Frobenius inner product of $A$ and $B$.

We shall denote by $\cH$ the class of functions $\tilde{v} \colon [0,T] \times \R^m \times \Sm \times \Gamma \to \R$ which explode asymptotically; that is, those functions $\tilde{v}$ such that
$$\tilde{v}(t,\mu,\Sigma,a) \,\longrightarrow\, \infty$$
as $|\mu| + \|\Sigma\| + \|a\| \to \infty$, and as $\lmin(\Sigma) \to 0$, and, in the case when $\rho$ is uncertain, as $\lmax(\rho\rho^\top) \to 1$.
\end{notation}

\begin{proposition}\label{propsolnsmoothHJB}
Under Assumption~\ref{filtassumptions}, the approximate value function $v^\eta$, as defined in \eqref{eq:defnvrobustsmooth}, is locally Lipschitz continuous with respect to $(t,\mu,\Sigma,a)$, and is the unique viscosity solution of the HJB equation
\begin{equation}\label{eq:filtsmoothHJB}
\frac{\pa v^\eta}{\pa t} + b_\mu\cdot\nabla_\mu v^\eta + b_\Sigma:\nabla_\Sigma v^\eta + \sup_{u \in U}\big\{h:\nabla_a v^\eta - f\big\} + (\lambda\cdot\nabla_\mu v^\eta - \psi)\dot{\eta}_t = 0
\end{equation}
in the class $\cH$ which satisfies the initial condition $v^\eta(0,\mu,\Sigma,a) = g(\mu,\Sigma,a)$.
\end{proposition}

\begin{proof}
As the path $\eta$ is smooth, the associated PDE \eqref{eq:filtsmoothHJB} is classical, except for the nonlinearities inherited from the filtering equations. As the proof of this result is lengthy, and not intended to be the focus of the current work, we will only give a sketch of the proof.

That $v^\eta$ is a viscosity solution of \eqref{eq:filtsmoothHJB} is a standard application of the dynamic programming principle; we refer to the proof of Proposition~4.9 in \cite{AllanCohen2019} for precise details.

Heuristically, as a result of \eqref{eq:fgrowth} and \eqref{eq:hboundinrho}, for terminal conditions $(\mu,\Sigma,a)$ which take extreme or close to degenerate values, i.e.~when either $|\mu| + \|\Sigma\| + \|a\| \gg 1$ or $\lmin(\Sigma) \approx 0$ or $\lmax(\rho\rho^\top) \approx 1$, it takes very expensive controls to allow the state trajectories $(q,R,\gamma)$ to escape these parts of their domain. It then follows from the growth conditions \eqref{eq:ggrowth}--\eqref{eq:ggrowthinrho} that the value function itself must explode as one approaches these extreme and degenerate values; that is, $v^\eta \in \cH$.

One can prove that $v^\eta$ is locally Lipschitz in all of its arguments by adapting the proof of Theorem~2.2 in \cite{BardiDaLio1997}, which in particular requires the strict inequality $\delta_2 > \delta_1$ in Assumption~\ref{filtassumptions}.

The controlled dynamics do not satisfy the standard Lipschitz condition which would be required to be able to apply a standard uniqueness result for Hamilton--Jacobi equations on unbounded domains, as in e.g.~Yong and Zhou \cite[Chapter~4]{YongZhou1999}. Nevertheless, uniqueness for an equation of the same form as \eqref{eq:filtsmoothHJB} was established in \cite[Section~5]{AllanCohen2019}, and an analogous argument may be used here. The main insight of this result is that the extra condition one should impose to obtain uniqueness is that solutions belong to the space $\cH$; that is, one should restrict to solutions which explode as they approach the boundary.
\end{proof}

The main result of this section is given by the following theorem.

\begin{theorem}
Under Assumption~\ref{filtassumptions}, the value function $v$, as defined in \eqref{eq:defnvrobust}, solves the rough HJB equation
\begin{equation*}
\rd v + \big(b_\mu\cdot\nabla_\mu v + b_\Sigma:\nabla_\Sigma v\big)\hspace{1pt}\rd t + \sup_{u \in U}\big\{h:\nabla_a v - f\big\}\hspace{0.3pt}\rd t + \lambda\cdot\nabla_\mu v\,\rd\zeta - \psi\,\rd \bmz = 0
\end{equation*}
with
$$v(0,\mu,\Sigma,a) = g(\mu,\Sigma,a)$$
in the sense of Definition~\ref{defnsolnroughHJB}. Moreover, writing $v = v^\zeta$, the map from $\sC^{0,p}_g([0,T];\R^d) \to \R$ given by $\bmz \mapsto v^\zeta(t,\mu,\Sigma,a)$ is locally uniformly continuous with respect to each of the rough path metrics $\varrho_p$ and $\varrho_{\opHol}$, locally uniformly in $(t,\mu,\Sigma,a)$.
\end{theorem}

\begin{proof}
Let $K$ be a compact subset of $\R^m \times \Sm \times \Gamma$ and let $\bme \in \sC^p$ be another rough path such that $\varrho_{\opHol}(\bme,\bmz) \leq 1$. By possibly replacing $L$ by $L + 1$, we may assume that $\ver{\bme}_{\opHol} \leq L$. Let us write $q^\eta$ (resp.~$q^\zeta$) for the solution of \eqref{eq:qdynamics} driven by $\eta$ (resp.~$\zeta$), and write $v^\eta$ (resp.~$v^\zeta$) for the corresponding value function, as defined in \eqref{eq:defnvrobust}.

By Corollary~\ref{corofiltrestrctrsoncpt}, there exists a constant $M > 0$ such that, for terminal conditions $(t,\mu,\Sigma,a) \in [0,T] \times K$, we may restrict to controls $u \in \cU^M \subseteq \cU$ such that
$$\|q^\eta\|_\infty,\ \|q^\zeta\|_\infty,\ \|R\|_\infty,\ \|\gamma\|_\infty,\ \|R\|_{\onevarzt},\ \|\gamma\|_{\onevarzt}$$
are all bounded by $M$. In the following we will allow the multiplicative constant indicated by the symbol $\lesssim$ to also depend on $M$.

By the Young--L\'oeve inequality (see e.g.~Theorem~6.8 in \cite{FrizVictoir2010}), we have
\begin{align*}
\bigg|\int_s^t(R_rc^\top_r\hspace{-1pt} + \sigma_r\rho_r)\,\rd(\eta - \zeta)_r\bigg| &\lesssim \big\|R_sc^\top_s\hspace{-1pt} + \sigma_s\rho_s\big\||(\eta_t - \zeta_t) - (\eta_s - \zeta_s)|\\
&\qquad + \big\|Rc^\top\hspace{-1pt} + \sigma\rho\big\|_{\onevarst}\|\eta - \zeta\|_{\pvarst}\\
&\lesssim \|\eta - \zeta\|_{\pvarst},
\end{align*}
from which we deduce that
$$|q^\eta_s - q^\zeta_s| \lesssim \int_s^t|q^\eta_r - q^\zeta_r|\,\rd r + \|\eta - \zeta\|_{\pvarst}$$
for all $s \in [0,t]$, and thus, by Gr\"onwall's inequality, that
\begin{equation}\label{eq:qsupbound}
\|q^\eta - q^\zeta\|_\infty \lesssim \|\eta - \zeta\|_{\pvarzt}.
\end{equation}

Since the state variables $q^\eta,q^\zeta,R$ and $\gamma$ are uniformly bounded, we are free to modify the coefficients $b_\mu,\lambda$ and $\psi$ outside of some large ball containing the domain of the state variables in its interior, without affecting the solutions $q^\eta,q^\zeta$. We may therefore pretend that actually $b_\mu \in \Lipb$ and $\lambda,\psi \in C^3_b$, so that in particular the hypotheses of Proposition~\ref{proproughstability} are satisfied. By the same argument, we may also suppose that $f$ and $g$ are Lipschitz in $q$.

By Proposition~\ref{proproughstability} combined with \eqref{eq:qsupbound}, we obtain
\begin{equation}\label{eq:filtintpsibound}
\bigg\|\int_0^\cdot\psi(q^\eta_s,\gamma_s)\,\rd\bme_s - \int_0^\cdot\psi(q^\zeta_s,\gamma_s)\,\rd\bmz_s\bigg\|_{\pvarzt} \lesssim \varrho_p(\bme,\bmz).
\end{equation}
Using \eqref{eq:qsupbound} and \eqref{eq:filtintpsibound}, we have, for any terminal condition $(t,\mu,\Sigma,a) \in [0,T] \times K$, that
\begin{align*}
\big|v^\eta&(t,\mu,\Sigma,a) - v^\zeta(t,\mu,\Sigma,a)\big|\\
&\leq \sup_{u \in \cU^M}\bigg|\int_0^t\big(f(q^\eta_s,R_s,\gamma_s,u_s) - f(q^\zeta_s,R_s,\gamma_s,u_s)\big)\,\rd s\\
&\qquad \qquad + \int_0^t\psi(q^\eta_s,\gamma_s)\,\rd\bme_s - \int_0^t\psi(q^\zeta_s,\gamma_s)\,\rd\bmz_s + g(q^\eta_0,R_0,\gamma_0) - g(q^\zeta_0,R_0,\gamma_0)\bigg|\\
&\lesssim \sup_{u \in \cU^M}\bigg(\int_0^t|q^\eta_s - q^\zeta_s|\,\rd s + \varrho_p(\bme,\bmz) + |q^\eta_0 - q^\zeta_0|\bigg)\\
&\lesssim \varrho_p(\bme,\bmz) \lesssim \varrho_{\opHol}(\bme,\bmz),
\end{align*}
and we conclude as we did in the proof of Theorem~\ref{thmvaluefuncsolvesroughHJB}.
\end{proof}

\begin{remark}
As we have seen, in order to prevent degeneracy of the control problem it is necessary to control the derivative of the parameters, rather than controlling them directly. This allows us to calibrate, not only beliefs about reasonable values the parameters could take, but also at what rate they should able to vary. For example, if one believes that the true parameters should remain fairly constant then one can put a large penalty on the magnitude of this derivative. In fact, by taking the penalty to be infinite for all non-zero controls (derivatives), we obtain a setting with unknown parameters which are constant in time. The discrete-time results of \cite{Cohen2017} suggest that we should then expect the resulting filter to converge to the true parameter. (Although our observations are not independent and identically distributed as in \cite{Cohen2017}, under reasonable conditions they are ergodic, and this leads to consistency properties in the likelihood function; see \cite{DoucMoulinesOlssonvanHandel2011}. In this case, we expect that this would lead to the nonlinear expectation asymptotically converging to the `true' expectation, and the analysis of \cite{Cohen2017} further suggests an interpretation of the nonlinear expectation in terms of confidence intervals.) Establishing precise convergence results could be the subject of future research.
\end{remark}

\vspace{12pt}
\noindent \textbf{Acknowledgements}\ \ A.L. Allan was supported by the Engineering and Physical Sciences Research Council [EP/L015811/1]. S.N. Cohen was supported by the Oxford-Man Institute for Quantitative Finance and the Oxford-Nie Financial Big Data Laboratory. The authors would also like to thank David J. Pr\"omel for helpful discussions.
\vspace{2pt}

\appendix

\section{Rough path estimates}\label{AppendixRoughPaths}

Before establishing existence of solutions to the RDE \eqref{eq:RDEforX}, we recall some useful estimates from Friz and Zhang \cite{FrizZhang2018}.

\begin{lemma}[Lemma~3.6 in \cite{FrizZhang2018}]\label{lemmainvariancebounds}
Let $\psi \in C^3_b$, $\gamma \in \Cptvar$ and $\bmz \in \sC^p$ with $\ver{\bmz}_p \leq L$. Let $(X,X') \in \sD^p_\zeta$. Then
$$\bigg(\int_0^\cdot\psi(X_s,\gamma_s)\,\rd\bmz_s,\psi(X,\gamma)\bigg) \in \sD^p_\zeta$$
is a controlled rough path, and we have
\begin{align*}
\big\|\psi(X,\gamma)\big\|_p &\leq C\Big(\big(|X'_0| + \|X'\|_p\big)\|\zeta\|_p + \big\|R^X\big\|_{\frac{p}{2}} + \|\gamma\|_{\frac{p}{2}}\Big),\\
\Big\|R^{\int_0^\cdot\psi(X_s,\gamma_s)\,\rd\bmz_s}\Big\|_{\frac{p}{2}} &\leq C\Big(1 + |X'_0| + \|X'\|_p + \big\|R^X\big\|_{\frac{p}{2}} + \|\gamma\|_{\frac{p}{2}}\Big)^{\hspace{-2pt}2}\ver{\bmz}_p,
\end{align*}
where the constant $C$ depends on $\psi,p$ and $L$.
\end{lemma}

The following lemma is a direct consequence of Lemma~3.4 in \cite{FrizZhang2018}.

\begin{lemma}\label{lemmacontractionbounds}
Let $\psi \in C^2_b$, $\gamma,\vartheta \in \Cptvar$ and $\bme,\bmz \in \sC^p$ with $\ver{\bme}_p,\ver{\bmz}_p \leq L$. Let $(X,X') \in \sD^p_\eta$ and $(Y,Y') \in \sD^p_\zeta$. For any $\delta \geq 1$, we have the following estimate
\begin{align*}
&\big\|\psi(X,\gamma) - \psi(Y,\vartheta)\big\|_p + \delta\Big\|R^{\int_0^\cdot\psi(X_s,\gamma_s)\,\rd\bme_s} - R^{\int_0^\cdot\psi(Y_s,\vartheta_s)\,\rd\bmz_s}\Big\|_{\frac{p}{2}}\\
&\leq C\bigg(\big\|R^{\psi(X,\gamma)} - R^{\psi(Y,\vartheta)}\big\|_{\frac{p}{2}} + \delta\Big(\big|\psi(X,\gamma)'_0\big| + \big\|\psi(X,\gamma)'\big\|_p + \big\|R^{\psi(X,\gamma)}\big\|_{\frac{p}{2}}\Big)\varrho_p(\bme,\bmz)\\
&\hspace{11pt} + \delta\Big(\big|\psi(X,\gamma)'_0 - \psi(Y,\vartheta)'_0\big| + \big\|\psi(X,\gamma)' - \psi(Y,\vartheta)'\big\|_p + \big\|R^{\psi(X,\gamma)} - R^{\psi(Y,\vartheta)}\big\|_{\frac{p}{2}}\Big)\ver{\bmz}_p\bigg),
\end{align*}
where the constant $C$ depends on $p$ and $L$.
\end{lemma}

\begin{lemma}[Lemma~3.5 in \cite{FrizZhang2018}]\label{lemmacompwithsmoothfuncbounds}
Let $\psi \in C^3_b$, $\gamma,\vartheta \in \Cptvar$ and $\bme,\bmz \in \sC^p$ with $\ver{\bme}_p$, $\ver{\bmz}_p \leq L$. Let $(X,X') \in \sD^p_\eta$ and $(Y,Y') \in \sD^p_\zeta$. Suppose that
$$|X'_0| + \|X'\|_p + \big\|R^X\big\|_{\frac{p}{2}} \leq M \qquad \text{and} \qquad |Y'_0| + \|Y'\|_p + \big\|R^Y\big\|_{\frac{p}{2}} \leq M$$
and $\|\gamma\|_{\frac{p}{2}},\|\vartheta\|_{\frac{p}{2}} \leq M$ for some $M > 0$. Then we have
\begin{align*}
\big\|\psi(X,\gamma)' - \psi(Y,\vartheta)'\big\|_p \leq C\Big(|X_0& - Y_0| + |X'_0 - Y'_0| + \|X' - Y'\|_p\\
&+ \big\|R^X - R^Y\big\|_{\frac{p}{2}} + \|\gamma - \vartheta\|_\infty + \|\gamma - \vartheta\|_{\frac{p}{2}} + \varrho_p(\bme,\bmz)\Big),
\end{align*}
\vspace{-18pt}
\begin{align*}
\big\|R^{\psi(X,\gamma)} - R^{\psi(Y,\vartheta)}\big\|_{\frac{p}{2}} \leq C\Big(|X_0& - Y_0| + |X'_0 - Y'_0| + \|X' - Y'\|_p\|\zeta\|_p\\
&+ \big\|R^X - R^Y\big\|_{\frac{p}{2}} + \|\gamma - \vartheta\|_\infty + \|\gamma - \vartheta\|_{\frac{p}{2}} + \varrho_p(\bme,\bmz)\Big),
\end{align*}
where the constant $C$ depends on $\psi,p,L$ and $M$.
\end{lemma}

\begin{proof}[Proof of Theorem~\ref{thmsolnofRDE}]
The following argument is adapted from the proof of Theorem~3.8 in \cite{FrizZhang2018}. Let $L > 0$ be such that $\ver{\bmz}_p \leq L$. We define a map $\cM^\gamma_T \colon \sD^p_\zeta \to \sD^p_\zeta$ by
$$\cM^\gamma_T(X,X') := \bigg(x + \int_0^\cdot b(X_s,\gamma_s)\,\rd s + \int_0^\cdot\lambda(X_s,\gamma_s)\,\rd\bmz_s,\,\lambda(X,\gamma)\bigg).$$
We will show that this map has a unique fixed point. For $\delta \geq 1$, we define the ball
$$\cB^{(\delta)}_T := \Big\{(X,X') \in \sD_\zeta^{p}([0,T];\R^m) : (X_0,X'_0) = (x,\lambda(x,\gamma_0)),\, \|X,X'\|_{\zeta,p}^{(\delta)} \leq 1\Big\},$$
where
$$\|X,X'\|_{\zeta,p}^{(\delta)} := \|X'\|_p + \delta\big\|R^X\big\|_{\frac{p}{2}}.$$
We will show that, for a suitable choice of $\delta$ and for $T$ sufficiently small, $\cM^\gamma_T$ leaves $\cB^{(\delta)}_T$ invariant, and then that it is a contraction on $\cB^{(\delta)}_T$.

By Lemma~\ref{lemmainvariancebounds}, any $(X,X') \in \cB^{(\delta)}_T$ satisfies
\begin{align*}
\big\|\cM^\gamma_T(X,X')\big\|_{\zeta,p}^{(\delta)} &\leq \big\|\lambda(X,\gamma)\big\|_p + \delta\bigg\|\int_0^\cdot b(X_s,\gamma_s)\,\rd s\bigg\|_{\frac{p}{2}} + \delta\Big\|R^{\int_0^\cdot\lambda(X_s,\gamma_s)\,\rd\bmz_s}\Big\|_{\frac{p}{2}}\\
&\leq C_1\bigg(\|\gamma\|_{\ptvarzT} + \delta\ver{\bmz}_{\pvarzT} + \delta T + \frac{1}{\delta}\bigg)
\end{align*}
for some constant $C_1 \geq \frac{1}{2}$ depending only on $b,\lambda,p,L$ and $\|\gamma\|_{\frac{p}{2}}$. Let $\delta = \delta_1 := 2C_1 \geq 1$, so that
$$\big\|\cM^\gamma_T(X,X')\big\|_{\zeta,p}^{(\delta_1)} \leq C_1\Big(\|\gamma\|_{\ptvarzT} + 2C_1\ver{\bmz}_{\pvarzT} + 2C_1T\Big) + \frac{1}{2}.$$
Hence, taking $T = T_1$ sufficiently small, we can ensure that $\|\cM^\gamma_{T_1}(X,X')\|_{\zeta,p}^{(\delta_1)} \leq 1$, so that $\cM^\gamma_{T_1}(X,X') \in \cB^{(\delta_1)}_{T_1}$. That is, $\cB^{(\delta_1)}_{T_1}$ is invariant under $\cM^\gamma_{T_1}$.

Let $(X,X'),(Y,Y') \in \cB^{(\delta_1)}_T$ for some $T \leq T_1$. For any (new) $\delta \geq 1$ we have
\begin{align*}
\big\|\cM^\gamma_T(X,X') - &\cM^\gamma_T(Y,Y')\big\|_{\zeta,p}^{(\delta)} \leq \delta\bigg\|\int_0^\cdot b(X_s,\gamma_s)\,\rd s - \int_0^\cdot b(Y_s,\gamma_s)\,\rd s\bigg\|_{\frac{p}{2}}\\
&\qquad + \big\|\lambda(X,\gamma) - \lambda(Y,\gamma)\big\|_p + \delta\Big\|R^{\int_0^\cdot\lambda(X_s,\gamma_s)\,\rd\bmz_s} - R^{\int_0^\cdot\lambda(Y_s,\gamma_s)\,\rd\bmz_s}\Big\|_{\frac{p}{2}}\\
&\leq C\bigg(\delta T\|X - Y\|_\infty + \big\|R^{\lambda(X,\gamma)} - R^{\lambda(Y,\gamma)}\big\|_{\frac{p}{2}}\\
&\qquad \quad + \delta\Big(\big\|\lambda(X,\gamma)' - \lambda(Y,\gamma)'\big\|_p + \big\|R^{\lambda(X,\gamma)} - R^{\lambda(Y,\gamma)}\big\|_{\frac{p}{2}}\Big)\ver{\bmz}_p\bigg).
\end{align*}
for some constant $C$ depending on $b,p$ and $L$, where we used the result of Lemma~\ref{lemmacontractionbounds} to obtain the last line.

We can take $M > 0$, dependent only on $\lambda$ and $\|\gamma\|_{\frac{p}{2}}$, sufficiently large such that $$\|\gamma\|_{\frac{p}{2}} \leq M \qquad \text{and} \qquad |X'_0| + \|X'\|_p + \big\|R^X\big\|_{\frac{p}{2}} \leq M$$
for all $(X,X') \in \cB^{(\delta_1)}_T$. Noting that $\|X - Y\|_\infty \leq \|R^X - R^Y\|_{\frac{p}{2}}$ and applying the estimates in Lemma~\ref{lemmacompwithsmoothfuncbounds}, we then deduce that
\begin{align*}
\big\|\cM^\gamma_T(X,X'&) - \cM^\gamma_T(Y,Y')\big\|_{\zeta,p}^{(\delta)}\\
&\leq C_2\bigg(\big\|R^X - R^Y\big\|_{\frac{p}{2}} + \delta\Big(\|X' - Y'\|_p + \big\|R^X - R^Y\big\|_{\frac{p}{2}}\Big)\big(\ver{\bmz}_{\pvarzT} + T\big)\bigg),
\end{align*}
for a new constant $C_2 > \frac{1}{2}$ which depends only on $b,\lambda,p,L$ and $M$. Let $\delta = \delta_2 := 2C_2 > 1$. We can then choose $T = T_2 \leq T_1$ sufficiently small such that $C_2\delta_2(\ver{\bmz}_{p;[0,T_2]} + T_2) \leq \frac{1}{2}$. We then have that
\begin{align*}
\big\|\cM^\gamma_{T_2}(X,X') - \cM^\gamma_{T_2}(Y,Y')\big\|_{\zeta,p}^{(\delta_2)} &\leq \frac{1}{2}\|X' - Y'\|_p + \frac{\delta_2 + 1}{2}\big\|R^X - R^Y\big\|_{\frac{p}{2}}\\
&\leq \frac{\delta_2 + 1}{2\delta_2}\big\|(X,X') - (Y,Y')\big\|_{\zeta,p}^{(\delta_2)},
\end{align*}
which establishes the contraction property for $\cM^\gamma_{T_2}$.

It follows that there exists a unique fixed point $(X,X') \in \sD_\zeta^{p}$ of the map $\cM^\gamma_{T_2}$, which is then the unique solution of \eqref{eq:RDEforXthm} in $\sD_\zeta^{p}$ satisfying $X' = \lambda(X,\gamma)$ over the time interval $[0,T_2]$. Noting that the time $T_2$ was chosen independently of the initial values $x$, $\gamma_0$, we may then simply paste solutions together to obtain a unique solution over the entire interval $[0,T]$ for any given $T > 0$.
\end{proof}

\begin{proof}[Proof of Proposition~\ref{proproughstability}]
Since $\|\gamma\|_{\frac{p}{2}}$ is bounded by $M$, it follows from Proposition~\ref{proproughestimates} and the fact that $X' = \lambda(X,\gamma)$, that there exists an $\tilde{M} > 0$, depending on $b,\lambda,p,T,L$ and $M$, such that the norms
$$\|\gamma\|_{\frac{p}{2}},\ \|X\|_p,\ |X'_0|,\ \|X'\|_p,\ \big\|R^X\big\|_{\frac{p}{2}},\ \big|\lambda(X,\gamma)'_0\big|,\ \big\|\lambda(X,\gamma)'\big\|_p,\ \big\|R^{\lambda(X,\gamma)}\big\|_{\frac{p}{2}},$$
and the same with $X$ and $\gamma$ replaced by $Y$ and $\vartheta$, are all bounded by $\tilde{M}$. In particular we note that the hypotheses of Lemma~\ref{lemmacompwithsmoothfuncbounds} are satisfied. In the following the symbol $\lesssim$ will denote inequality up to a multiplicative constant which may depend on $b,\lambda,\psi,p,T,L$ and $\tilde{M}$.

For any $\delta \geq 1$, we have
\begin{align*}
\|X' - &Y'\|_p + \delta\big\|R^X - R^Y\big\|_{\frac{p}{2}}\\
&\lesssim \delta\bigg\|\int_0^\cdot b(X_s,\gamma_s)\,\rd s - \int_0^\cdot b(Y_s,\vartheta_s)\,\rd s\bigg\|_{\frac{p}{2}}\\
&\hspace{35pt} + \big\|\lambda(X,\gamma) - \lambda(Y,\vartheta)\big\|_p + \delta\Big\|R^{\int_0^\cdot\lambda(X_s,\gamma_s)\,\rd\bme_s} - R^{\int_0^\cdot\lambda(Y_s,\vartheta_s)\,\rd\bmz_s}\Big\|_{\frac{p}{2}}.
\end{align*}
Since the drift $b$ is Lipschitz, it is easy to see that
\begin{align*}
\bigg\|\int_0^\cdot b(X_s,\gamma_s)\,\rd s - \int_0^\cdot b(Y_s,\vartheta_s)\,\rd s\bigg\|_{\frac{p}{2}} &\lesssim \big(\|X - Y\|_\infty + \|\gamma - \vartheta\|_\infty\big)T\\
&\leq \big(|x - y| + \|X - Y\|_p + \|\gamma - \vartheta\|_\infty\big)T.
\end{align*}
As $(X,X') = (X,\lambda(X,\gamma)) \in \sD^p_\eta$ and $(Y,Y') = (Y,\lambda(Y,\vartheta)) \in \sD^p_\zeta$, we have that
\begin{align}
\|X - Y\|_p &\leq \|X'\|_\infty\|\eta - \zeta\|_p + \|X' - Y'\|_\infty\|\zeta\|_p + \|R^X - R^Y\|_{\frac{p}{2}}\nonumber\\
&\lesssim \|\eta - \zeta\|_p + |x - y| + \|\gamma - \vartheta\|_\infty + \|X' - Y'\|_p + \|R^X - R^Y\|_{\frac{p}{2}}.\label{eq:XYpbasicbound}
\end{align}
Combining the results of Lemmas~\ref{lemmacontractionbounds} and \ref{lemmacompwithsmoothfuncbounds}, we then deduce that
\begin{align*}
\|X' &- Y'\|_p + \delta\big\|R^X - R^Y\big\|_{\frac{p}{2}}\\
&\leq C_0\bigg(|x - y| + \delta\varrho_p(\bme,\bmz) + \big\|R^X - R^Y\big\|_{\frac{p}{2}} + \|\gamma - \vartheta\|_\infty + \|\gamma - \vartheta\|_{\frac{p}{2}}\\
&\hspace{35pt} + \delta\Big(|x - y| + \|X' - Y'\|_p + \big\|R^X - R^Y\big\|_{\frac{p}{2}}\\
&\hspace{130pt}+ \|\gamma - \vartheta\|_\infty + \|\gamma - \vartheta\|_{\frac{p}{2}}\Big)\Big(\ver{\bmz}_{\pvarzT} + T\Big)\bigg),
\end{align*}
for some constant $C_0 > \frac{1}{2}$ which depends on $b,\lambda,p,T,L$ and $\tilde{M}$.

Let $\delta = \delta_0 := 2C_0 > 1$. We can then take $T = T_0$ (depending only on $p,L$ and $C_0$) sufficiently small such that
$$C_0\delta_0\Big(\ver{\bmz}_{p;[0,T_0]} + T_0\Big) \leq C_0\delta_0\Big(\|\zeta\|_{\opHol}T_0^{\frac{1}{p}} + \big\|\zeta^{(2)}\big\|_{\tpHol}T_0^{\frac{2}{p}} + T_0\Big) \leq \frac{1}{2},$$
so that, after rearranging, we obtain
\begin{align*}
\|X' - Y'&\|_p + (\delta_0 - 1)\big\|R^X - R^Y\big\|_{\frac{p}{2}}\\
&\leq (\delta_0 + 1)\Big(|x - y| + \|\gamma - \vartheta\|_\infty + \|\gamma - \vartheta\|_{\frac{p}{2}}\Big) + \delta_0^2\varrho_p(\bme,\bmz).
\end{align*}
It follows that the estimate in \eqref{eq:roughstability} holds over any time interval of length $T_0$. One can then extend this estimate to hold over the union of any finite number of such intervals (with a correspondingly larger constant $C$) by pasting via Lemma~\ref{lemmapvarpartitionbound}.

The bound in \eqref{eq:XYpbasicbound} also holds with $X$ and $Y$ replaced with $\int_0^\cdot\psi(X_s,\gamma_s)\,\rd\bme_s$ and $\int_0^\cdot\psi(Y_s,\vartheta_s)\,\rd\bmz_s$ respectively, so that
\begin{align*}
\bigg\|\int_0^\cdot\psi(X_s,\gamma_s)\,\rd\bme_s - \int_0^\cdot\psi(Y_s,\vartheta_s)\,\rd&\bmz_s\bigg\|_p \lesssim \|\eta - \zeta\|_p + |x - y| + \big\|\psi(X,\gamma) - \psi(Y,\vartheta)\big\|_p\\
&+ \|\gamma - \vartheta\|_\infty + \Big\|R^{\int_0^\cdot\psi(X_s,\gamma_s)\,\rd\bme_s} - R^{\int_0^\cdot\psi(Y_s,\vartheta_s)\,\rd\bmz_s}\Big\|_{\frac{p}{2}}.
\end{align*}
Applying again the results of Lemmas~\ref{lemmacontractionbounds} and \ref{lemmacompwithsmoothfuncbounds}, this time with $\delta = 1$, we deduce that
\begin{align*}
\bigg\|\int_0^\cdot&\psi(X_s,\gamma_s)\,\rd\bme_s - \int_0^\cdot\psi(Y_s,\vartheta_s)\,\rd\bmz_s\bigg\|_p\\
&\lesssim |x - y| + \|X' - Y'\|_p + \big\|R^X - R^Y\big\|_{\frac{p}{2}} + \|\gamma - \vartheta\|_\infty + \|\gamma - \vartheta\|_{\frac{p}{2}} + \varrho_p(\bme,\bmz).
\end{align*}
Combining this with \eqref{eq:roughstability}, we obtain \eqref{eq:roughstabilityintegral}.
\end{proof}

\bibliographystyle{abbrv}
\bibliography{References_Andy}

\end{document}